\documentclass[11pt,a4paper]{amsart}
\title{Morse predecomposition of an invariant set}
\usepackage[utf8]{inputenc}
\usepackage[english]{babel}
\usepackage{amsaddr}
\usepackage{amsmath}
\usepackage{amsthm}
\usepackage{amssymb}
\usepackage{hyperref}
\usepackage{comment}
\usepackage[shortlabels]{enumitem}
\usepackage{subfiles}
\usepackage{faktor} 

\usepackage{algorithm}
\usepackage{algorithmicx}
\usepackage{algpseudocode}

\algnewcommand{\Null}{\mathbf{null}}
\algnewcommand{\True}{\mathbf{true}}
\algnewcommand{\False}{\mathbf{false}}

\def\refeq#1{\if\workingver y(\ref{#1})-[[#1]]\else(\ref{#1})\fi}
\def\refth#1{\if\workingver y\ref{#1}-[[#1]]\else\ref{#1}\fi}
\def\mylabel#1{\if\workingver y\label{#1}{\bf\ \ [[#1]]\ \ }\else\label{#1}\fi}
\def\mybibitem#1{\if\workingver y\bibitem{#1}{\bf\ \ [[#1]]\ \
}\else\bibitem{#1}\fi}

\newfont{\msam}{msam10}
\newfont{\msbm}{msbm10}

\def\articletheorems{
\newtheorem{thm}{Theorem}[section]
\newtheorem{lem}[thm]{Lemma}

\newtheorem{defn}[thm]{Definition}
\newtheorem{cor}[thm]{Corollary}
\newtheorem{prop}[thm]{Proposition}
\newtheorem{ex}[thm]{Example}
\newtheorem{algo}{Algorithm}[section] 
\newtheorem{alg}[thm]{Algorithm}  

}

\def\cA{\text{$\mathcal A$}}
\def\cB{\text{$\mathcal B$}}

\def\cE{\text{$\mathcal E$}}

\def\cM{\text{$\mathcal M$}}
\def\cN{\text{$\mathcal N$}}

\def\cP{\text{$\mathcal P$}}
\def\cQ{\text{$\mathcal Q$}}

\def\cS{\text{$\mathcal S$}}
\def\cT{\text{$\mathcal T$}}
\def\cU{\text{$\mathcal U$}}
\def\cV{\text{$\mathcal V$}}

\newcommand{\cl}{\operatorname{cl}}

\newcommand{\Int}{\operatorname{int}}
\newcommand{\inte}{\operatorname{int}}

\newcommand{\dom}{\operatorname{dom}}

\newcommand{\mo}{\operatorname{mo}}

\newcommand{\im}{\operatorname{im}}

\renewcommand{\emptyset}{\varnothing}

\newcommand{\Inv}{\operatorname{Inv}}

\def\proof{{\bf Proof:\ }}

\def\begeq#1{\begin{equation}\mylabel{#1}}
\def\endeq{\end{equation}}

\def\mathobj#1{\mbox{$#1$}}

\def\NN{\mathobj{\mathbb{N}}}

\def\PP{\mathobj{\mathbb{P}}}

\def\RR{\mathobj{\mathbb{R}}}

\def\WW{\mathobj{\mathbb{W}}}
\def\ZZ{\mathobj{\mathbb{Z}}}

\def\implies{\;\Rightarrow\;}

\def\setof#1{\mbox{$\{\,#1\,\}$}}

\def\0#1{\hbox{\kern25pt}$ #1 $\\}
\def\1#1{\hbox{\kern40pt}$ #1 $\\}
\def\2#1{\hbox{\kern55pt}$ #1 $\\}
\def\3#1{\hbox{\kern70pt}$ #1 $\\}

\newcounter{li}

\def\begalg#1{\begin{algo}\mylabel{#1}\normalshape:\small\baselineskip 10pt\\}
\def\endalg{\end{algo}}

\def\Figures(include=#1,cat=#2){
  \renewcommand{\textfraction}{.20}
  \renewcommand{\topfraction}{.80}
  \renewcommand{\bottomfraction}{.80}
  \renewcommand{\floatpagefraction}{.80}
  \newcount\figcount
  \figcount=0
  \let\includefigures=#1
  \def\figcat{#2}
}

\def\FigureFromFile[#1][#2](#3)#4
{
  \begin{figure}[htbp]
     \global\advance\figcount by 1
     \if\includefigures y\special{anisoscale #1.wmf, \the\hsize #2}\fi
     \vspace{#2}
     \caption{#4}
     \mylabel{#3}
   \end{figure}
}

\def\FigureFromFileTwoD[#1][#2,#3](#4)#5
{
  \begin{figure}[htbp]
     \global\advance\figcount by 1
     \if\includefigures y\special{anisoscale #1.wmf, #2 #3}\fi
     \vspace{#2}
     \caption{#5}
     \mylabel{#4}
   \end{figure}
}

\def\FigureF<#1>[#2](#3)#4
{
  \begin{figure}[htbp]
     \global\advance\figcount by 1
     \if\includefigures y\special{anisoscale \figcat/fig\number\figcount.wmf,
       \the\hsize #2}
     \fi
     \if\includefigures p
       \leavevmode
       \epsfxsize=\hsize
       \epsffile{#1}
     \fi
     \if\includefigures y
          \vspace{#2}
     \fi
     \caption{#4}
     \mylabel{#3}
   \end{figure}
}

\def\Figure[#1](#2)#3
{
  \begin{figure}[htbp]
     \global\advance\figcount by 1
     \if\includefigures y\special{anisoscale \figcat/fig\number\figcount.wmf,
       \the\hsize #1}
     \fi
     \if\includefigures p
       \leavevmode
       \epsfxsize=\hsize
       \epsffile{fig\number\figcount.eps}
     \fi
     \if\includefigures y
          \vspace{#1}
     \fi
     \caption{#3}
     \mylabel{#2}
   \end{figure}
}

\usepackage[square,numbers]{natbib}
\theoremstyle{definition}
\articletheorems

\usepackage{graphbox}
\usepackage{graphicx}
\graphicspath{{figures/}}

\newcommand{\pto}{\nrightarrow}
\newcommand{\inv}{\operatorname{Inv}}
\newcommand{\inter}{\operatorname{int}}

\renewenvironment{proof}{{\bfseries Proof:\ }}{\vspace{10pt}}

\newcommand{\supsetV}[1]{\left[#1\right]_\cV}
\newcommand{\supconvsetV}[1]{\big\langle#1\big\rangle_\cV}
\newcommand{\esol}{\operatorname{eSol}}

\def\VSet{\mathbb{P}}
\def\VSubSet{{\mathbb{Q}}}
\def\VSetV{\mathbb{V}}
\def\VSetBis{\cQ}

\usepackage{hyperref}

\newcommand{\exend}{\hspace*{\fill}$\Diamond$}

\author{Micha\l{} Lipi\'nski}
\address{Intitute of Science and Technology Austria,\\
         Am Campus 1, 3400 Klosterneuburg, Austria}
\email{michal.lipinski@ist.ac.at}

\author{Konstantin Mischaikow}
\address{Department of Mathematics and BioMaPS Institute,\\
         Rutgers University,
         Piscataway, NJ 08854, USA}
\email{mischaik@math.rutgers.edu}

\author{Marian Mrozek}
\address{Division of Computational Mathematics,
         Faculty of Mathematics and Computer Science,
         Jagiellonian University,
         ul.~St. \L{}ojasiewicza 6, 30-348~Krak\'ow, Poland}
\email{marian.mrozek@uj.edu.pl}

\begin{document}

%
\thanks{
	Research of M.M. is partially supported by the Polish National Science Center under Opus Grant No. 2019/35/B/ST1/00874.
	M.L. acknowledge support by the Dioscuri program initiated by the Max Planck Society, jointly managed with the National Science Centre (Poland), and mutually funded by the Polish Ministry of Science and Higher Education and the German Federal Ministry of Education and Research.
M.L. also acknowledge that this project has received funding from the European Union’s Horizon 2020 research and innovation programme under the Marie Skłodowska-Curie Grant Agreement No. 101034413.
	   The work of K.M. was partially supported by the National Science Foundation under awards DMS-1839294 and HDR TRIPODS award CCF-1934924, DARPA contract HR0011-16-2-0033, National Institutes of Health award R01 GM126555,  Air Force Office of Scientific Research under award numbers FA9550-23-1-0011, AWD00010853-MOD002 and MURI FA9550-23-1-0400.  K.M. was also supported by a grant from the Simons Foundation.
}
\date{Version compiled on \today}

\begin{abstract}
Motivated by the study of the recurrent orbits in a Morse set of a Morse decomposition,
we introduce the concept of Morse predecomposition of an isolated invariant set in the setting of combinatorial and classical dynamical systems. 
We prove that a Morse predecomposition indexed by a poset is a Morse decomposition
and we show how a Morse predecomposition may be condensed back to a Morse decomposition.

\end{abstract}

\maketitle

\section{Introduction}

A Morse decomposition of a flow on a compact metric space is a finite collection of disjoint, compact invariant sets, referred to as {\em Morse sets}, such that all the recurrent dynamics of the flow is contained in the union of the Morse sets and the Morse sets can be indexed by a partially ordered set (poset).
Morse decompositions provide a rigorous, computable approximation to C. Conley's fundamental result on the decomposition of a compact invariant set into its chain recurrent part and gradient-like part  \cite{Conley1978,Conley1988}.
Identification of nontriviality of Morse sets can be done via the Conley index \cite{Conley1978}, but more detailed understanding of the structure of the dynamics within a Morse sets requires additional information. 
With this in mind, in this paper we seek to extend the concept of a Morse decomposition to a Morse predecomposition that provides a more refined representation of the internal structure of Morse sets.
In essence, we give up the partial order condition and we weaken the concept of a connection between two isolated invariant sets.
We present the definition both in the combinatorial setting of multivector fields~\cite{Mr2017,LKMW2020} (see Definition~\ref{def:mvf-mpd}) and in the classical setting of flows (see Definition~\ref{defn:predecflow}).
We prove that a Morse predecomposition with a partial order as the preorder is always a Morse decomposition (see Corollary~\ref{cor:preMorse=Morse} and Theorem~\ref{thm:morse-decomposition}) and show how, in the combinatorial case, one can reconstruct a Morse decomposition from a given Morse predecomposition  (see Theorem~\ref{thm:mvf_consolidation}).

Our presentation is motivated by the fact that algorithms for constructing Morse decompositions are based on a combinatorial representation of dynamics in the form of a multivalued map acting on a collection of cells decomposing the phase space. 
The multivalued map sends a cell $Q$ into a collection of cells that captures the forward evolution of $Q$ under the  dynamics. 
The action of the multivalued map may be interpreted as a directed graph with the collection of cells as vertices.
The strongly connected components of this graph identify isolating blocks and the associated isolated invariant
sets form the constructed Morse decomposition.

In the context of discrete dynamics, i.e., dynamics generated by a continuous function $f\colon X\to X$, there is extensive work and broadly applicable software for identification of Morse decompositions \cite{AraiEtAl2009, BushAtAl2012, bush:cowan:harker:mischaikow, liz:pilarczyk} and analysis of the structure of dynamics within the Morse sets \cite{mischaikow:mrozek:95, szymczak, day:junge:mischaikow, day:frongillo, PilSigGra2023}.
Comparable tools are much less developed in the context of flows, e.g., dynamics generated by a differential equation \cite{WSLLK2023}.
For the purposes of this paper we focus on the theory of combinatorial multivector fields~\cite{Mr2017,LKMW2020,MrWa2023,DeLiMrSl2023} (an extension of the methods of combinatorial topological dynamics introduced by R.~Forman \cite{Fo98a,Fo98b}) for which the above mentioned graph theoretic algorithms produce Morse decompositions and for which identification of finer structure of dynamics within Morse sets is possible, e.g., the identification of periodic orbits and chaotic dynamics \cite{MWS2022}.
We expect that effective extension of these latter results will require  combinatorial decompositions of strongly connected components and a deeper understanding of the dynamic interpretations of these decompositions, which are subjects of this paper.  

To understand the idea of a Morse predecomposition on a simple example, consider the flow visualized in Figure~\ref{fig:flow_example_intro}(top)
and the isolated invariant set $S$ highlighted in pink.
We recall that a {\em Morse decomposition} of $S$ is an indexed family $\cM=\{M_p\mid p\in\PP\}$ of mutually disjoint isolated invariant sets 
(called Morse sets) with a partial order $(\PP, \leq)$ such that for every point $x\in S\setminus\bigcup\cM$ 
we can find $p,q\in\PP$ such that $q<p$, $\alpha(x)\subset M_p$ and $\omega(x)\subset M_q$.
The finest Morse decomposition of $S$  
consists of three Morse sets, $M_b=\{B\}$, $M_d=\{D\}$ and set $M_o$ composed of stationary points $A$, $C$, $E$, 
together with the heteroclinic orbits connecting them.
The connections between the Morse sets are presented in Figure~\ref{fig:flow_example_intro}(bottom left) in the form of an acyclic directed graph (digraph).
The Morse decomposition cannot capture the stationary points $A$, $C$, $E$ and the heteroclinic connections between them.

We define a {\em Morse predecomposition} of $S$ as an indexed family $\cM'=\{M_p\mid p\in\PP'\}$ of mutually disjoint, closed invariant sets 
such that $\alpha(x)\cap M_p\neq\emptyset$ and $\omega(x)\cap M_q\neq\emptyset$ for some $p,q\in\PP'$ where $\PP'$ is an arbitrary finite set.
With such a definition we can split $M_o$ into three sets $M_a$, $M_c$ and $M_e$ corresponding to stationary points $A$, $C$, and $E$ 
as independent elements of $\cM'$.
Thus, the resulting Morse predecomposition is $\cM'=\setof{M_b,M_d,M_{a},M_{b},M_{c}}$.  
The corresponding digraph representing the connections  is presented in Figure~\ref{fig:flow_example_intro}(bottom right).
We call it the {\em Conley model} of $\cM'$ (see Section \ref{subsec:predecomposition_continuous}).

The directed graph presented in Figure~\ref{fig:str-conn-comp} has four strongly connected components marked in light green. 
There are also three strongly connected sets which are not strongly connected components.
They are marked in orange.
Every strongly connected set is contained in a strongly connected component. 
However, a strongly connected component may have many non-empty, proper, disjoint, strongly connected subsets. 
Hence, even on a purely combinatorial level the internal structure of a strongly connected component may vary.

Under weak conditions, every strongly connected component of a digraph obtained from a cellular decomposition of the phase space of a flow identifies an isolating neighborhood and a Morse set inside. 
This raises the question whether a strongly connected subset of a strongly connected component also identifies an isolating neighborhood and an isolated invariant set. 
This is not true in general, but in the context of the combinatorial multivector field models discussed above we can provide conditions under which the answer is positive. 
Thus, we first introduce Morse predecompositions for multivector fields where the concept arises naturally, is relatively simple, and suggests how to set up the analogous definition for flows.  

This paper is organized as follows.
Section \ref{sec:preliminaries} contains preliminaries. 
In Section~\ref{sec:dynsys}
we introduce basic concepts and facts concerning dynamics, both classical and combinatorial.
The combinatorial formulation of Morse predecomposition, as well as its properties are presented in Section \ref{sec:pred-mvf}.
Section \ref{sec:predecomposition_continuous} contains the definition of the Morse predecomposition with the related results 
for classical dynamical systems.
Concluding remarks together with an outline of the further direction of the research are discussed 
in the final Section \ref{sec:concluding_remarks}.

\begin{figure}
  \includegraphics[width=0.9\textwidth]{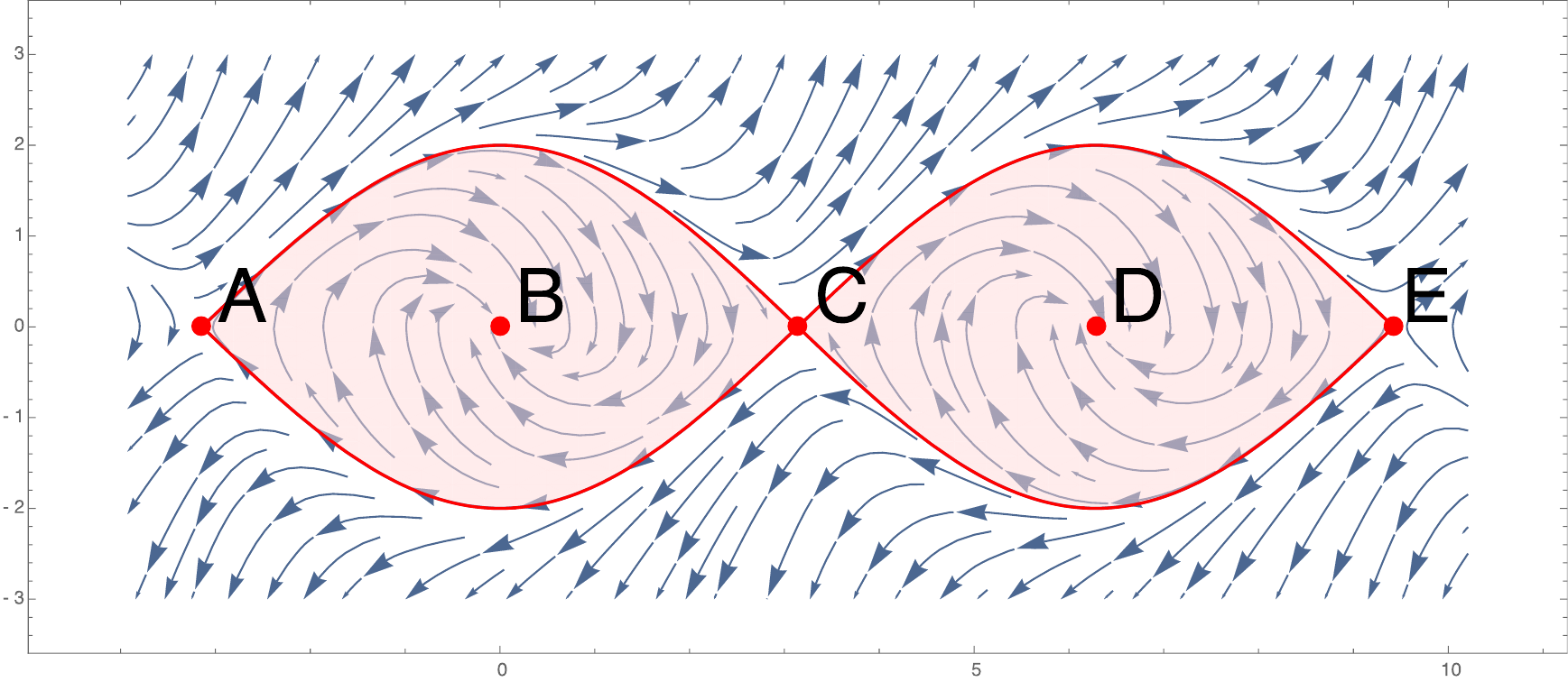}

  \includegraphics[height=2.5cm]{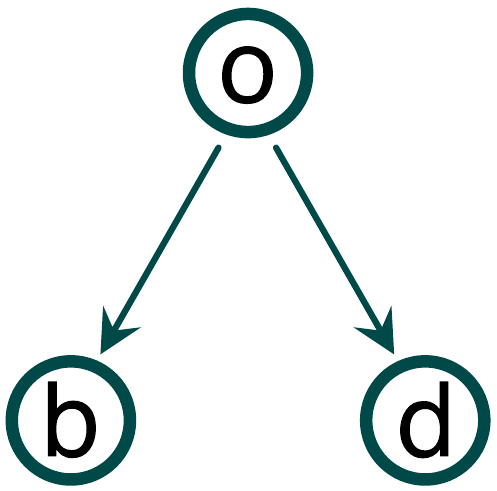}
  \hspace{1cm}
  \includegraphics[height=2.5cm]{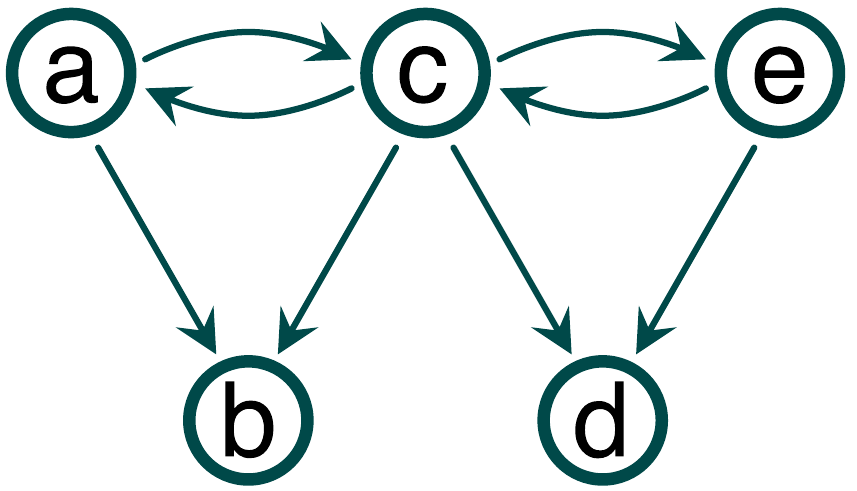}
  \caption{Top: An isolated invariant set $S$ marked in pink. 
      Bottom left: The flow defined Conley model for the minimal Morse decomposition of $S$.
      Bottom right: The flow defined Conley model for the minimal Morse predecomposition of $S$.
  }
  \label{fig:flow_example_intro}
\end{figure}

\begin{figure}
  \includegraphics[width=0.5\textwidth]{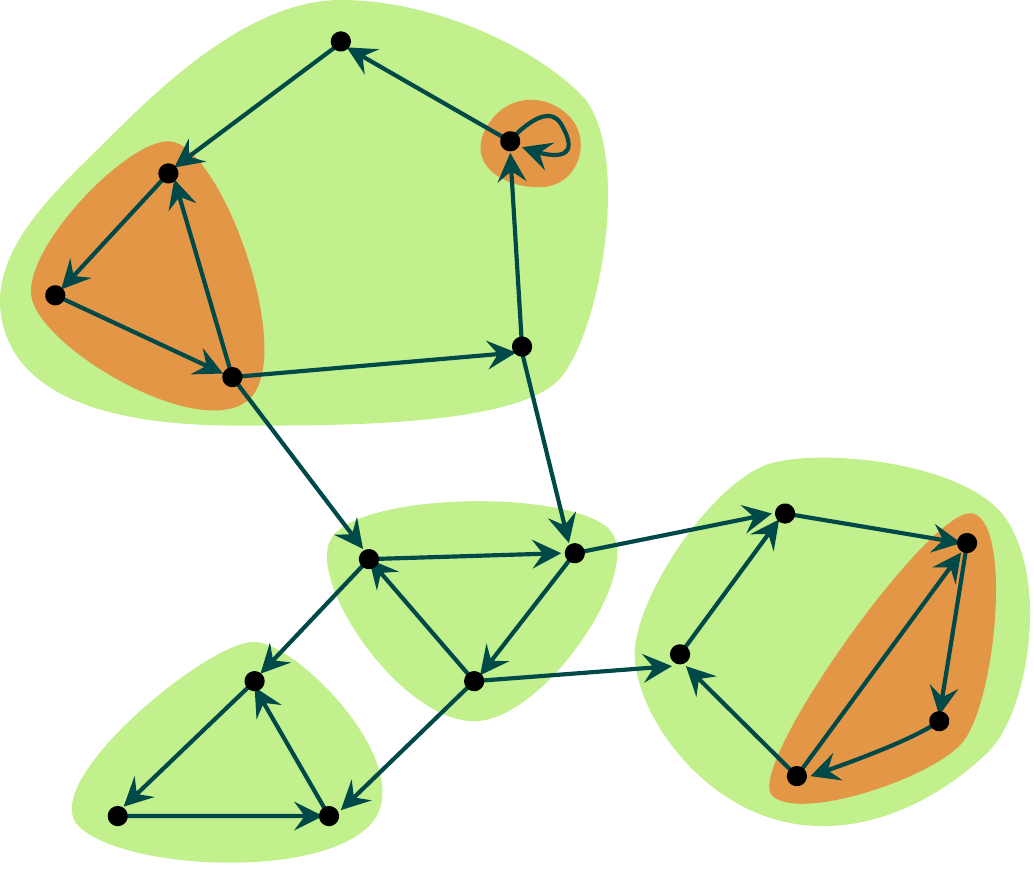}
  \caption{A directed graph with four strongly connected components marked in light green. The three sets marked in orange are strongly connected
  but not strongly connected components. 
  }
  \label{fig:str-conn-comp}
\end{figure}

\section{Preliminaries}
\label{sec:preliminaries}

This section gathers notation, terminology and basic results needed in the sequel.

\subsection{Sets, families and maps}
\label{sec:sets-families-maps}
We denote  the sets of real numbers, non-negative real numbers,  integers, non-negative integers and natural numbers respectively 
by~$\RR$, $\RR^+$ $\ZZ$, $\ZZ^+$, and $\NN$. 

By a {\em family} we mean a set whose elements are sets.  
Given a set~$X$, we denote the family of all subsets of~$X$ by~$\cP(X)$.

An {\em indexed family} is, formally speaking, a function assigning to every $\iota\in J$ the set $A_\iota$ (see~ \cite[Section I.1]{En1989}).
Informally, we will write an indexed family as  
$\{A_\iota\mid\iota\in J\}$.
Every family $\cA$ may be consider an indexed family $\{A\mid A\in\cA\}$. 
The converse is not true. Although an indexed family is often identified with the family of its values, there is a subtle difference. The family of values may be smaller, because $A_\iota$ and $A_\kappa$ may be equal as sets, but they are different in the indexed family if $\iota\neq\kappa$. 
In the case when the elements of a family are mutually disjoint, the difference between families and indexed families matters only in the case of the empty set, which may appear only once in a family but many times in an indexed family.

A {\em partition} of a set~$X$ is a family $\cE\subset\cP(X)$
of mutually disjoint, non-empty subsets of~$X$ such that $\bigcup\cE=X$.
Given two families $\cA,\cB\subset\cP(X)$ we say that $\cA$ is {\em inscribed} in $\cB$
or $\cA$ is a refinement of $\cB$
if for each $A\in\cA$ there is a $B\in\cB$ such that $A\subset B$.

We write $f:X\pto Y$ for a {\em partial map} from $X$ to $Y$, that is, a map
defined on a subset $\dom f\subset X$, called the {\em domain} of $f$,
and such that the set of values of $f$, denoted by
$\im f$, is contained in $Y$. 

We write $F:X\multimap Y$ to denote a {\em multivalued map}, that is a map
defined on $X$ with subsets of $Y$ as values. 
We identify such an $F$ with the relation
\[
\setof{(x,y)\in X\times Y\mid y\in F(x)}.
\]
Such an identification lets us define the {\em composition} $G\circ F$ of multivalued maps $F:X\multimap Y$ and $G:Y\multimap Z$
as the composition of relations. 
In particular, we are interested in the iterates of a multivalued map $F:X\multimap X$
considered as a multivalued dynamical system. 

\subsection{Relations}
\label{sec:relations}
Let $X$ be a set and let $R\subset X\times X$ be a binary relation in $X$.
Relation $R'\subset X\times X$ is an {\em extension} of relation $R$ if $R\subset R'$.
A set $X$ with a reflexive and transitive relation $R$ is called a \emph{preordered set}; relation $R$ is called a \emph{preorder}.
If $R$ is additionally antisymmetric we call it a \emph{partial order}, and $X$ a \emph{partially ordered set} (or a poset).
The total relation $X\times X$ is clearly a preorder containing every binary relation $R$ in $X$.
Hence, the family of preorders extending given $R$ is not empty. 
The intersection of this family is easily seen to be the smallest preorder in $X$ extending $R$.
We call it the {\em preorder induced by $R$} and denote it $\leq_R$.
The following proposition is straightforward.
\begin{prop}
\label{prop:po-extension}
If the preorder $\leq_R$ induced by relation $R\subset X\times X$ admits an extension which is a partial order, 
then $\leq_R$ itself is a partial order. 
\qed
\end{prop}

\subsection{Digraphs}
\label{sec:digraphs}
Note that the pair $(X,R)$ may be considered a {\em directed graph} (a {\em digraph}) 
with $X$ as its set of {\em vertices} and  $R$ as its set of {\em edges}.
Vice versa, given a digraph $(V,E)$, the collection of edges $E$ is a relation on the set of vertices $V$.
This lets us  freely mix the terminology typically used for relations with the terminology used for digraphs.

Given a digraph $G=(V,E)$, by a {\em walk in $G$ from $v$ to $w$ of length $k>0$}
we mean a sequence $v=v_0,v_1,\ldots v_k=w$ of vertices in $V$ such that $(v_{i-1},v_{i})\in E$ for $i=1,2,\ldots k$.
Such a walk is a {\em cycle} if $v_0=v_k$. A cycle of length one is a {\em loop}.
Digraph is {\em acyclic} if it has no cycles other then loops.
A subset $A\subset V$ of vertices is {\em strongly connected} if for any $v,w\in A$ there is a walk 
$v=v_0,v_1,\ldots v_k=w$ such that $v_i\in A$ for $i\in\{0,1,\ldots, k\}$.
A vertex $v\in V$ is strongly connected if $\{v\}$ is strongly connected which happens if and only if there is a loop at $v$.
A {\em strongly connected component} of $G$ is a strongly connected subset of vertices which is maximal 
with respect to inclusion. Note that two strongly connected components of $G$ are either disjoint or equal
but, in general,  the family of all strongly connected components need not be a partition of $V$.

By mixing the terminology of relations and digraphs,
we say that $G=(V,E)$ is {\em reflexive} if $E$ considered as a relation in $V$ is reflexive
or, equivalently, if there is a loop at every vertex $v$. Note that the family of strongly connected components
of a reflexive digraph $G=(V,E)$ is a partition of the vertex set $V$.
Given a reflexive digraph $G=(V,E)$, we write $\leq_G$ for the \emph{preorder in $V$ induced by $E$}.
We note that the preorder $\leq_G$ of an acyclic digraph $G$ is a always a partial order.

\subsection{Topological spaces}
\label{sec:top}
For the basic terminology concerning topological
spa\-ces we refer the reader to~\cite{En1989,munkres:00a}.
In terms of notation, given a topological space $X$, we denote by $\inte A$ the {\em interior} of~$A\subset X$
and by $\cl A$ the closure of $A$. 
A subset~$A\subset X$ is {\em locally closed} if $\mo A:=\cl A\setminus A$ is closed 
(see~\cite[Problem 2.7.1]{En1989} for several equivalent conditions for local closedness).

A topological space~$X$ is a {\em finite topological space} if the underlying set~$X$ is finite. 
In this paper we consider finite topological spaces satisfying $T_0$ separation axiom.
By Alexandrov Theorem~\cite{Al1937} every finite, $T_0$ topological space $(X,\cT)$ can be uniquely identified with a finite poset $(X, \leq_\cT)$ defined by 
\begin{equation}\label{eq:poset_topology}
    x\leq_\cT y\ \Longleftrightarrow\ x \in \cl y.
\end{equation}
In particular, a subset $A$ of a finite topological space is locally closed if and only if 
it is \emph{convex} with respect to the poset, that is if $x\leq y\leq z$ and $x,z\in A$ imply $y\in A$.

\section{Dynamical systems}
\label{sec:dynsys}

In this section we gather basic concepts and results in dynamics needed in the paper. 
We first recall classical definitions and statements for flows and then we introduce
the concept of a combinatorial multivector field and the associated combinatorial dynamical system. 
Note that our notation and terminology for combinatorial systems shadows that for flows. 
We do so to emphasize the similarities. 

\subsection{Flows}
Assume that $(X, d)$ is a locally compact metric space and $\varphi: X\times \RR\rightarrow X$ is a flow on $X$. 

We recall that a partial map $\gamma:\RR\pto X$ is a {\em solution} of $\varphi$ if $\gamma(s+t)=\varphi(\gamma(s),t)$
for all $s,s+t\in\dom\gamma$. 
A {\em solution through $x\in X$} is a solution such that $x\in\im\gamma$.
A {\em solution in $S\subset X$} is a solution such that $\im\gamma\subset S$.
A solution $\gamma$ is {\em full} if $\dom\gamma=\RR$; otherwise it is {\em partial}.

We recall that the \emph{maximal invariant set} of $N\subset X$ is given by
\[
\Inv(N) = \Inv(N, \varphi) := \setof{ x\in N\mid \varphi(\RR,x)\subset N}.
\]
We say that a set $S\subset X$ is \emph{invariant} if $\Inv S = S$.
A compact set $N\subset X$ is an \emph{isolating neighborhood} if $\Inv N \subset \Int N$.
An invariant set $S\subset X$ is \emph{an isolated invariant set} if 
there exists an isolating neighborhood $N$ such that $\Inv N = S$.

A fundamental feature of isolating neighborhoods is that they persist under a small perturbation. 
More precisely, we have the following proposition.

\begin{prop}
\label{prop:perturbation}
\cite[Proposition 1.1]{MischMroz2002}
Given  a continuously parametrized family of dynamical systems
\[
    \varphi_\lambda:\RR\times X\rightarrow X,\quad \lambda\in[-1,1]
\]
and an isolating neighborhood for $\varphi_0$, there is an $\varepsilon>0$ such that
$N$ is also an isolating neighborhood for $\varphi_\lambda$ with $\lambda\in(-\varepsilon,\varepsilon)$. 
\qed
\end{prop}

The \emph{omega limit set of a point $x\in X$}, denoted by $\omega(x)$, is the maximal invariant set contained in the closure of the forward solution passing through $x$: 
\begin{align*}
  \omega(x) &:= \bigcap_{t>0}\cl\varphi(x, [t,+\infty)).
\end{align*}
Symmetrically, we have the \emph{alpha limit set of $x$}, denoted by $\alpha(x)$ defined as
\begin{align*}
  \alpha(x) &:= \bigcap_{t<0}\cl\,\varphi(x,(-\infty,t]).
\end{align*}

As point $x$ uniquely determines the corresponding solution $\gamma(t):=\varphi(x, t)$ we use also notation $\alpha(\gamma)$ and $\omega(\gamma)$.
This will allow us to be more consistent with the notation used for multivalued combinatorial dynamical systems.

Let $S\subset X$ be a compact invariant set. 
An invariant set $A\subset S$ is an \emph{attractor in $S$} if there exists a neighborhood $U$ of $A$ such that $\omega(U\cap S) = A$. A compact invariant set $R\subset S$ is a \emph{repeller in $S$} if there exists a neighborhood $V$ of $R$ such that $\alpha(V\cap S)=R$. If $S=X$ we just say that $A$ is an attractor (repeller).

\subsection{Chain recurrence}\label{subsec:chains}

Let $S$ be a compact metric space.
Fix an open cover $\cU$ of $S$ and a $T>0$.
A $(\cU, T)$-\emph{chain from $x$ to $y$} of length $n$ is a sequence $x=x_0,x_1,...,x_{n}=y$
such that there exist sequences $t_1,t_2,\ldots,t_n$ of reals numbers and $U_1, U_1, \ldots, U_n\in\cU$ 
satisfying $t_i\geq T$ and $x_i, \varphi(x_{i-1}, t_i)\in U_i$
for each $i=1,2\ldots n$.
Let us denote by $R(S)$ the set of points $x\in S$ such that for every open cover $\cU$ of $S$, and every $T>0$ there exists a non-constant $(\cU,T)$-chain in $S$ from $x$ to itself.
A set $A\subset S$ is \emph{chain recurrent} if $R(A)=A$.

As the consequence of \cite[4.1.D]{Conley1978} and \cite[6.3.C]{Conley1978} and dual results for $\alpha$ limit sets we obtain the following proposition. 
\begin{prop}
\label{prop:limits_are_chain_recurrent}
  For every point $x\in S$ limit sets $\alpha(x)$ and $\omega(x)$ are non-empty, compact, invariant, connected and chain recurrent.
\end{prop}

By \cite[Theorem 4.12(5)]{Akin1993} we immediately get the following property. 
\begin{prop}\label{prop:subset_of_limset_is_not_attr}
  If $C\subset S$ is a connected, invariant, chain recurrent set and there exists 
  a non-empty isolated invariant set $A$ such that $A\varsubsetneq C$, then $A$ is neither an attractor nor a repeller in $C$.
\end{prop}

The following lemma may be proved analogously to \cite[Theorem 6.2]{Churchill1971}.
\begin{lem}\label{lem:exist_alpha_omega_point}
  Assume $S$ is connected and $A\varsubsetneq S$ is a non-empty, isolated invariant set. 
  If $A$ is not an attractor (respectively not a repeller), 
  then there exists an $x\in S\setminus A$ such that $\alpha(x)\subset A$ (respectively $\omega(x)\subset A$).
\end{lem}

\begin{thm}\label{thm:exist_alpha_omega_point}
Assume  $C\subset S$ is a connected, invariant, chain recurrent set and $A$ is a non-empty isolated invariant set in $C$.
If $A\varsubsetneq C$, then there exist points $x, y\in C\setminus A$ such that $\alpha(x)\subset A$ and $\omega(y)\subset A$. 
\end{thm}
\begin{proof}
It follows from Proposition \ref{prop:subset_of_limset_is_not_attr} 
that $A$ is neither an attractor nor a repeller in $C$.
Hence, the conclusion follows  immediately from Lemma~\ref{lem:exist_alpha_omega_point}.
\qed
\end{proof}

\subsection{Combinatorial dynamical systems}\label{subsec:cds}
Let $X$ be a finite, $T_0$ topological space.
A {\em combinatorial dynamical system} on $X$ induced by a multivalued map $\Pi:X\multimap X$ is a multivalued map
$\Pi:X\times\ZZ^+\multimap X$ defined recursively for $x\in X$ and $n\in\ZZ^+$ by $\Pi(x,0):=\{x\}$ and $\Pi(x,n+1):=\Pi(\Pi(x,n))$.
Formally, the map $\Pi:X\multimap X$ is the {\em generator} of the combinatorial dynamical system but, to simplify the terminology, 
we often refer to $\Pi:X\multimap X$ as the combinatorial dynamical system. 

A \emph{solution} of $\Pi$ is a partial map $\gamma:\ZZ\pto X$ such that $\dom\gamma$ is an interval in $\ZZ$ 
and $\gamma(t+1)\in\Pi\left(\gamma(t)\right)$ for $t,t+1\in\dom\gamma$.
We recall that  $\Pi$ may be considered as a binary relation in $X$, hence, 
also as a directed graph $G_\Pi$.
In terms of the digraph $G_\Pi$, a solution may be identified with a walk in $G_\Pi$.
If $\dom\gamma$ is finite, we call $\gamma$ a \emph{path}. 
If $\dom\gamma=\ZZ$, we call $\gamma$  a \emph{full solution}.
We denote the restriction of a full solution $\gamma$ to the intervals $[p,q]$, $[0,\infty)$, and $(-\infty,0]$ 
respectively by $\gamma_{[p,q]}$, $\gamma^+$, and $\gamma^-$.

A \emph{shift} is a map $\tau_s:\ZZ\ni t\mapsto t+s$.
Let $\gamma$ and $\psi$ be solutions such that $\min\dom\gamma=p$ and $\max\dom\psi=q$ with $p,q\in\ZZ$.
Moreover, assume that $\gamma(p)\in\Pi(\psi(q))$.
Then, we define the {\em composition} of $\gamma$ and $\psi$ by $\gamma\cdot\psi:=\gamma\cup(\psi\circ\tau_{q-p})$.
It is straightforward to observe that the composition of solutions is also a solution.

\subsection{Combinatorial dynamics induced by multivector fields}\label{subsec:mvf}
In this subsection we briefly recall basic concepts of combinatorial dynamics generated by multivector fields. 
For a comprehensive introduction into multivector fields see \cite{LKMW2020}.

A \emph{multivector field} on $X$ is a partition $\cV$ of $X$ into locally closed subsets of $X$ (see Sec.~\ref{sec:top} for the definition of local closedness).
We refer to the elements of the partition as \emph{multivectors}.
Note that $\mo V=\cl V\setminus V$ is always closed for a multivector $V$, because  a multivector is locally closed.
We say that multivector $V$ is \emph{critical} if the relative singular homology $H(\cl V, \mo V)$ is non-trivial.
Otherwise, we say that $V$ is \emph{regular}.

For $A\subset X$ we set 
\[
\cV_A:=\{V\in\cV\mid V\cap A\neq\emptyset\}\quad\text{ and }\quad\supsetV{A}:=\bigcup\cV_A.
\]
We say that $A$ is \emph{$\cV$-compatible} if $A=\supsetV{A}$.
In particular, for every $x\in X$ set $\supsetV{x}:=\supsetV{\{x\}}$ is the unique multivector $V\in\cV$ containing $x$.

A multivector field $\cV$ induces a combinatorial dynamical system generated by $\Pi_\cV:X\multimap X$ 
given for an $x\in X$ by
\begin{equation}\label{eq:mvf_map}
  \Pi_\cV(x):=\cl x\cup[x]_\cV.
\end{equation}
As mentioned $\Pi_\cV$ induces a directed graph $G_{\Pi_\cV}$.
We will denote it for short as $G_\cV$.

By a solution of $\cV$ we mean a solution of the combinatorial dynamical system $\Pi_\cV$.
For a full solution $\gamma$ we define combinatorial \emph{limit sets} as
\begin{align}\label{eq:mvf_limit_sets}
  \alpha(\gamma):= \supconvsetV{\bigcap_{t\in{\mathbb{Z}^-}} \gamma((-\infty,t])}  
  \quad\text{   and   }\quad
  \omega(\gamma):= \supconvsetV{\bigcap_{t\in{\mathbb{Z}^+}} \gamma([t,\infty))},
\end{align}
where $\supconvsetV{A}$ denotes the \emph{$\cV$-hull of $A$}, that is the smallest locally closed and $\cV$-compatible set containing~$A$.

We say that a full solution $\gamma$ is \emph{non-essential} if either $\alpha{(\gamma)}$ or $\omega{(\gamma)}$ is contained in a single regular multivector; 
otherwise we say that $\gamma$ is \emph{essential}. These definitions of non-essential and essential solutions are easily seen to be equivalent to the 
definitions given in~\cite{LKMW2020}.
We denote by $\esol(A)$ the set of all essential solutions in $A$.
Additionally we set
\begin{align*}
  \esol(x, A):= \{\gamma\in\esol(A)\mid \gamma(0)=x\}.
\end{align*}
and we define the \emph{maximal invariant subset of} $S\subset X$ as
\[
  \inv S:=\{x\in S\mid\esol(x, S)\neq\emptyset\}.
\]
We say that set $S$ is {\em invariant} if $\inv S=S$.
A closed set $N$ is an \emph{isolating set} for an invariant set $S$ 
if $\Pi_\cV(S)\subset N$ and every path in $N$ with endpoints in $S$ is contained in $S$.
If an invariant set admits an isolating set we call it an \emph{isolated invariant set}.
Note that the combinatorial notion of isolation is not an exact counterpart of the classical one.
In particular,  the same set can be an isolating set for two different isolated invariant sets.
The need for a modification of isolation in finite topological spaces
comes from the tightness of such spaces caused by the lack of Hausdorff axiom. 

\begin{thm}
  \cite[Propositions 4.10 and 4.13,Theorem 4.13]{LKMW2020}
  \label{thm:lcl_vcomp_iis}
  Assume $S\subset X$ is invariant. Then $S$ is an isolated invariant set if and only if
  $S$ is  locally closed and $\cV$-compatible.
\end{thm}

As an immediate consequence of Theorem~\ref{thm:lcl_vcomp_iis} we get the following corollary. 
\begin{cor}
\label{cor:induced-mvf}
Assume $S$ is an isolated invariant set of $\cV$. Then $\cV_S:=\setof{V\in\cV\mid V\subset S}$
is a multivector field on $S$. We call it the {\em induced} multivector field.
\qed
\end{cor}

\begin{thm}
  \cite[Theorem 6.15]{LKMW2020}
  \label{thm:mvf-limit_set_iis}
  Let $\gamma\in\esol(X)$.
  Then both $\alpha(\gamma)$ and $\omega(\gamma)$ are non-empty, connected, isolated invariant sets.
  Moreover, sets $\alpha(\gamma)$ and $\omega(\gamma)$ are strongly connected with respect to $G_\cV$.
\end{thm}

The following proposition follows immediately from the definition of a strongly connected set. 
\begin{prop}
\label{prop:strongly-connected-surjection}
  Assume $C\subset X$ is strongly connected with respect to $G_\cV$.
  Then, there exists a path $\rho:[0,n]_\mathbb{Z}\to C$ such that $\im\rho=C$ and $\rho(0)=\rho(n)$.
\end{prop}

\begin{prop}
\label{prop:limit_solution}
  Let $\gamma\in\esol(X)$.
  Then there exist $\rho,\rho'\in\esol(\alpha(\gamma))$ such that $\im\rho=\alpha(\gamma)$ and $\im\rho'=\omega(\gamma)$. 
\end{prop}

\begin{proof}
  By Theorem \ref{thm:mvf-limit_set_iis} set $\alpha(\gamma)$ is strongly connected with respect to $G_\cV$.
  Hence, by Proposition~\ref{prop:strongly-connected-surjection} we can construct a path $\theta:[0,n]_\mathbb{Z}\to C$ such that $\im\theta =\alpha(\gamma)$ and $\theta(0)=\theta(n)$.
  Define $\rho:=\ldots\cdot\theta\cdot\theta\cdot\theta\cdot\ldots$.
  Note that  $\rho$ is essential, because 
  otherwise $\alpha(\gamma)$ would consist of only a single regular multivector which contradicts the fact that $\alpha(\gamma)$ is an isolated invariant set.
  Moreover, one easily verifies that $\alpha(\rho)=\omega(\rho)=\im\rho=\alpha(\gamma)$.
  The argument for $\omega(\gamma)$ is analogous.
  \qed
\end{proof}

\subsection{Morse decompositions in classical and combinatorial dynamics}
\label{subsec:md}
We now recall the classical concept of Morse decomposition for flows and its counterpart for multivector fields. 
Later in the paper, after introducing Morse predecompositions, we will give different but equivalent
definitions of Morse decomposition, formulated in terms of Morse predecompositions. 

Consider an isolated invariant set $S$ of a flow $\varphi$ on a locally compact metric space $X$.
The following definition of Morse decomposition, in the formulation by J.~Reineck~\cite[Definition 1.3]{reineck:90a},
goes back to R.~Franzosa~\cite{Fr1986}.
\begin{defn}
\label{defn:flow-Morse-decomposition-orig}
An indexed family  $\cM=\{M_p\mid p\in\VSet\}$ of mutually disjoint isolated invariant subsets of $S$
indexed by a poset  $\VSet$ is  a {\em Morse decomposition} of $S$ if for every $x \in S$ 
either $x \in M_p$ for some $p\in\VSet$ or there exist $p,q\in\VSet$ 
satisfying $p> q$, $\alpha(x)\subset M_p$ and $\omega(x)\subset M_q$.
\end{defn}

Consider now an isolated invariant set $S$ of a combinatorial multivector field $\cV$ on a finite topological space $X$.
The following definition of Morse decomposition for multivector fields was introduced in~\cite[Definition 7.1]{LKMW2020} 
\begin{defn}
\label{defn:mvf-Morse-decomposition-orig}
An indexed family  $\cM=\{M_p\mid p\in\VSet\}$ of mutually disjoint isolated invariant subsets of $S$
indexed by a poset  $\VSet$ is  a {\em Morse decomposition} of $S$ 
if for every essential solution $\gamma$ in $X$ 
either $\im\gamma\subset M_p$ for some $p\in\VSet$ or 
there exist $p,q\in\VSet$ 
satisfying $p> q$, $\alpha(\gamma)\subset M_p$ and $\omega(\gamma)\subset M_q$.
\end{defn}

As pointed out in~\cite{KMV2022}, Morse decomposition is an extremely general tool, but this generality has led to subtle variances
in its definition. Some approaches exclude the empty set as a Morse set, some allow for many empty Morse sets.
The latter approach requires a properly set up definition but has many benefits, 
in particular in the context of the stability of various invariants associated with Morse decomposition.
In~\cite{KMV2022}, the issue is resolved by differentiating Morse representation from Morse decomposition. 
This paper is based on the original definition by Franzosa~ \cite{Fr1984} in which Morse decomposition is considered as a properly understood 
indexed family. An indexed family also allows for many empty Morse sets, as pointed out in~ Section~\ref{sec:sets-families-maps}.

\section{Morse predecomposition in combinatorial dynamics}
\label{sec:pred-mvf}

In this section we introduce Morse predecompositions in the setting of combinatorial multivector fields.
We assume in this section that $\cV$ is a multivector field on a finite topological space $X$ and $X$ is invariant 
with respect to $\cV$. Such an assumption is not restrictive, because if $X$ is not invariant, we can replace $X$
by its invariant part and $\cV$ by its restriction to the maximal invariant subset. 

\subsection{Combinatorial Morse predecomposition}

In order to generalize the concept of Morse decomposition we first introduce the following definition classifying full solutions
which keep visiting isolated invariant set $S_1$ backward in time and isolated invariant set $S_2$ forward in time.

\begin{defn}
\label{defn:mvf-links}
Consider invariant sets $S_1$, $S_2$ of a multivector field $\cV$ in $X$. 
A full solution $\gamma:\ZZ\rightarrow X$ is a {\em link from $S_1$ to $S_2$} if $\alpha(\gamma)\cap S_1\neq\emptyset\neq\omega(\gamma)\cap S_2$. 
Link $\gamma$ is {\em homoclinic} if $S_1=S_2$. Otherwise it is {\em heteroclinic}.
A homoclinic link $\gamma$  is {\em trivial} if $\im\gamma\subset S_1=S_2$. 
Otherwise it is {\em non-trivial}.
\end{defn}

\begin{prop}
  \label{prop:ess_path_extension}
  Let $S_1,S_2\subset X$ be isolated invariant sets.
    There exists a link $\gamma$ from $S_1$ to $S_2$
    if and only if 
    there exists a path $\rho:[a,b]\rightarrow X$ such that $\rho(a)\in S_1$ and $\rho(b)\in S_2$.
\end{prop}
\begin{proof}
    Let $\gamma$ be a link from $S_1$ to $S_2$.
    Let $a,b\in\ZZ$ such that $\gamma(a)\in\alpha(\gamma)$ and $\gamma(b)\in\omega(\gamma)$.
    By Theorem \ref{thm:mvf-limit_set_iis} for any $x_1\in S_1\cap\alpha(\gamma)$ there exist a path $\rho_1$ from $x_1$ to $\gamma(a)$.
    Similarly we can find $x_2\in S_2\cap\omega(\gamma)$ a path $\rho_2$ from $\gamma(b)$ to $x_2$.
    Hence, $\rho:=\rho_1\cdot\gamma_{[a,b]}\cdot\rho_2$ is a path from $S_1$ to $S_2$.
    
    To prove the opposite implication consider a path $\rho:[a,b]\rightarrow X$ 
    such that $\rho(a)\in S_1$ and $\rho(b)\in S_2$. Take $\varphi_1 \in \esol(\rho(a), S_1)$ and $\varphi_2 \in \esol(\rho(b), S_2)$
    and define essential solution $\gamma:=\varphi_1^-\cdot\rho\cdot\varphi_2^+$.
    We get $\alpha(\gamma)=\alpha(\varphi_1)\subset S_1$ and $\omega(\gamma)=\omega(\varphi_2)\subset S_2$.
    Hence, $\gamma$ is a link from $S_1$ to $S_2$.
  \qed
\end{proof}

\begin{defn}
  \label{def:mvf-mpd}
  An indexed  family $\cM = \setof{M_p\mid p\in\VSet}$ of mutually disjoint isolated invariant subsets in $X$
  is a \emph{Morse predecomposition} of $X$ if every essential solution $\gamma$ in $X$ is a link from $M_p$ to $M_q$
  for some $p,q\in\VSet$.
  We refer to the sets in $\cM$ as \emph{pre-Morse sets}.
  We also define a Morse predecomposition of an isolated invariant set $S\subset X$ as a Morse predecomposition 
  with respect to the induced multivector field $\cV_S$.
\end{defn}

\begin{prop}
    Let $\cM$ be a Morse predecomposition of $X$.
    Then, for every non-empty isolated invariant set $S$ we have $S\cap\bigcup\cM\neq\emptyset$.
    In particular, if $V\in\cV$ is a critical multivector then $V\subset\bigcup\cM$.
\end{prop}

\begin{proof}
    Assume to the contrary that $S$ is a non-empty isolated invariant set such that $S\cap\bigcup\cM=\emptyset$.
    Since $S$ is non-empty, we can find a $\varphi\in\esol(S)$.
    Clearly, $\alpha(\varphi)\subset S$ and $\omega(\varphi)\subset S$ because of Theorem~\ref{thm:lcl_vcomp_iis} and \eqref{eq:mvf_limit_sets}.
    Thus, $\varphi$ is an essential solution while not being a link which contradicts the assumption that $\cM$ is a Morse predecomposition.
    \qed
\end{proof}

Links of a system forms a relation $L$ on an indexing set $\VSet$ of a Morse predecomposition $\cM$ defined as
\[
    L:=\{(p,q)\in \VSet\times\VSet \mid \text{there exists a non-trivial link from $M_p$ to $M_q$}  \}.
\]
The digraph $G_\cM:=(\VSet, L)$ is referred to as \emph{the digraph induced by Morse predecom\-posi\-tion $\cM$}.

We say that a preorder $\leq$ in $\VSet$ is {\em admissible} if the existence of a link from $M_p$ to $M_q$ implies $q\leq p$.
Clearly, the preorder induced by $G_\cM$ is the minimal admissible preorder.
We denote it by  $\leq_\cM$ and we call it the {\em flow induced preorder}.

A \emph{Conley model of $\cM$} is a digraph $C=(N, E)$ 
with a map $\nu_{G_\mathcal{M},C}: \VSet \rightarrow N$
such that it preserves the graph induced preorders.
Clearly, $G_\cM$ itself is a Conley model.
We call $G_\cM$ the \emph{flow-defined Conley model of \cM}.

\subsection{Combinatorial Morse decomposition}
In this section we present a definition of Morse decomposition formulated as a special case of Morse predecomposition. 
We later prove (see Theorem~\ref{thm:mvf-morse-decomposition}) that this definition is equivalent to Definition~\ref{defn:mvf-Morse-decomposition-orig}.
We begin with the definition of a connection, a special case of a link. 
\begin{defn}
\label{defn:mvf-connections}
Given a Morse predecomposition $\cM$, a link $\gamma$ from $M_p$ to $M_q$ is a {\em connection} if $\im\gamma\cap\bigcup\cM\subset M_p\cup M_q$
and $\alpha(\gamma)\subset M_p$, $\omega(\gamma)\subset M_q$.
Clearly, a trivial link is a homoclinic connection. However, note that not every homoclinic connection is a trivial link. 
\end{defn}

We say that an invariant subset $T$ of an invariant set $S$ is {\em saturated in $S$} if 
every homoclinic connection $\gamma$ from $T$ to $T$ in $S$ is trivial.
We say that a Morse predecomposition $\cM$ is {\em saturated} if every $M\in\cM$ is saturated. 

\begin{defn}
\label{defn:mvf-Morse-decomposition}
A Morse predecomposition  $\cM=\{M_p\mid p\in\VSet\}$ is called a {\em Morse decomposition} if it is saturated
and among admissible preorders there is a partial order. 
Note that then, by Proposition~\ref{prop:po-extension}, the flow induced preorder is a partial order. 
\end{defn}

\begin{ex}
\label{ex:mvf-1}
Consider a simplicial complex $K$ in Figure \ref{fig:mvf-example}.
Let $\cV$ be a multivector field on $K$ consisting of seven multivectors:
\begin{gather*}
  V_1:=\{A,AB,AD\},\quad
  V_2:=\{B,BC\},\quad
  V_3:=\{C,AC, CE\},\\
  V_4:=\{ACD\},\quad
  V_5:=\{D,CD\},\quad
  V_6:=\{BCE\},\\
  V_7:=\{E, BE, DE\}.
\end{gather*}
Multivectors $V_2$ and $V_5$ are regular, all others are critical.
The finest possible Morse decomposition consists of three isolated invariant sets $M_1=V_4$, $M_2=V_6$, and $M_3=K\setminus V_4\cup V_6$.
The Hasse diagram of the admissible partial order coincides with the $\cM$ induced digraph and is presented in Figure~\ref{fig:mvf-example} (top right).
On the other hand, Morse predecomposition allows us to treat every critical multivector contained in $M_3$ as a separate pre-Morse set.
Hence, we get a Morse predecomposition $\cM$ consisting of five sets: $M_1=V_4$, $M_2=V_6$, $M_3^A=V_1$, $M_3^C=V_3$, and $M_3^E=V_7$.
The induced digraph is presented in Figure~\ref{fig:mvf-example}(bottom right).
\exend
\end{ex}

\begin{figure}
  \centering
  \includegraphics[width=0.5\textwidth]{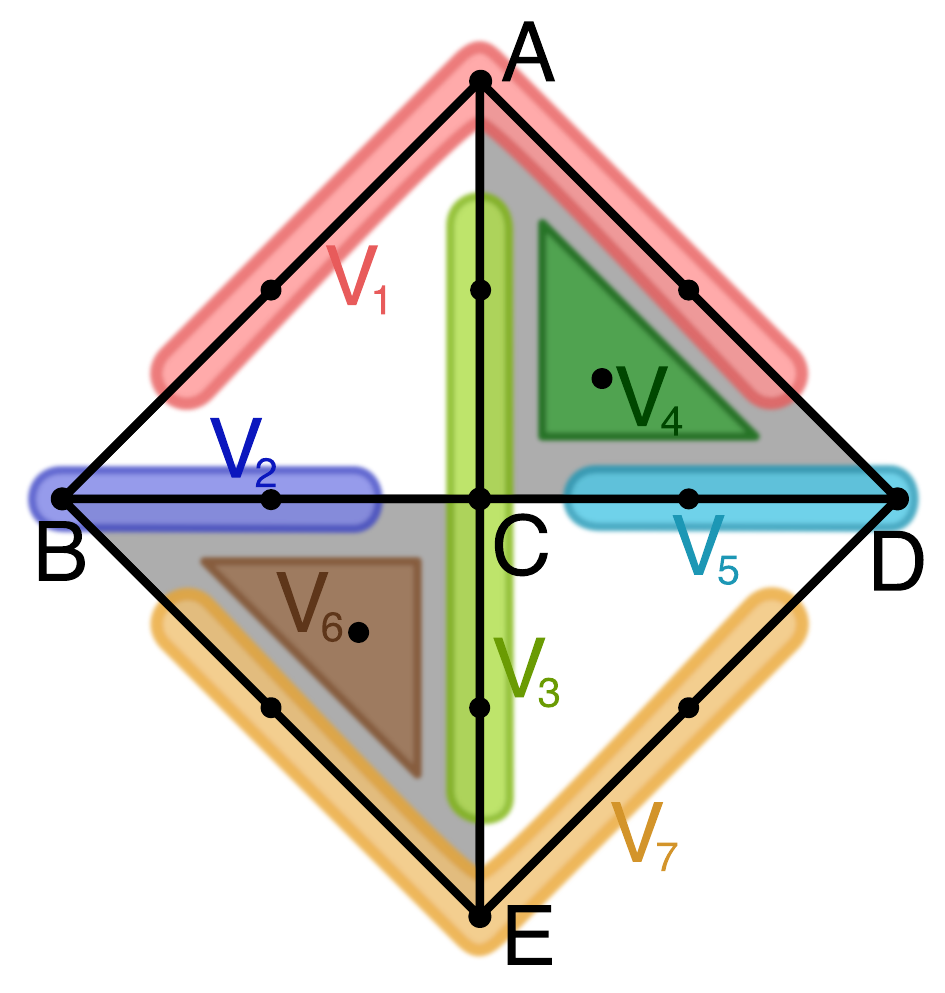}
  \hspace{0.01cm }
    \raisebox{0.05\height}{
        \includegraphics[width=0.4\textwidth]{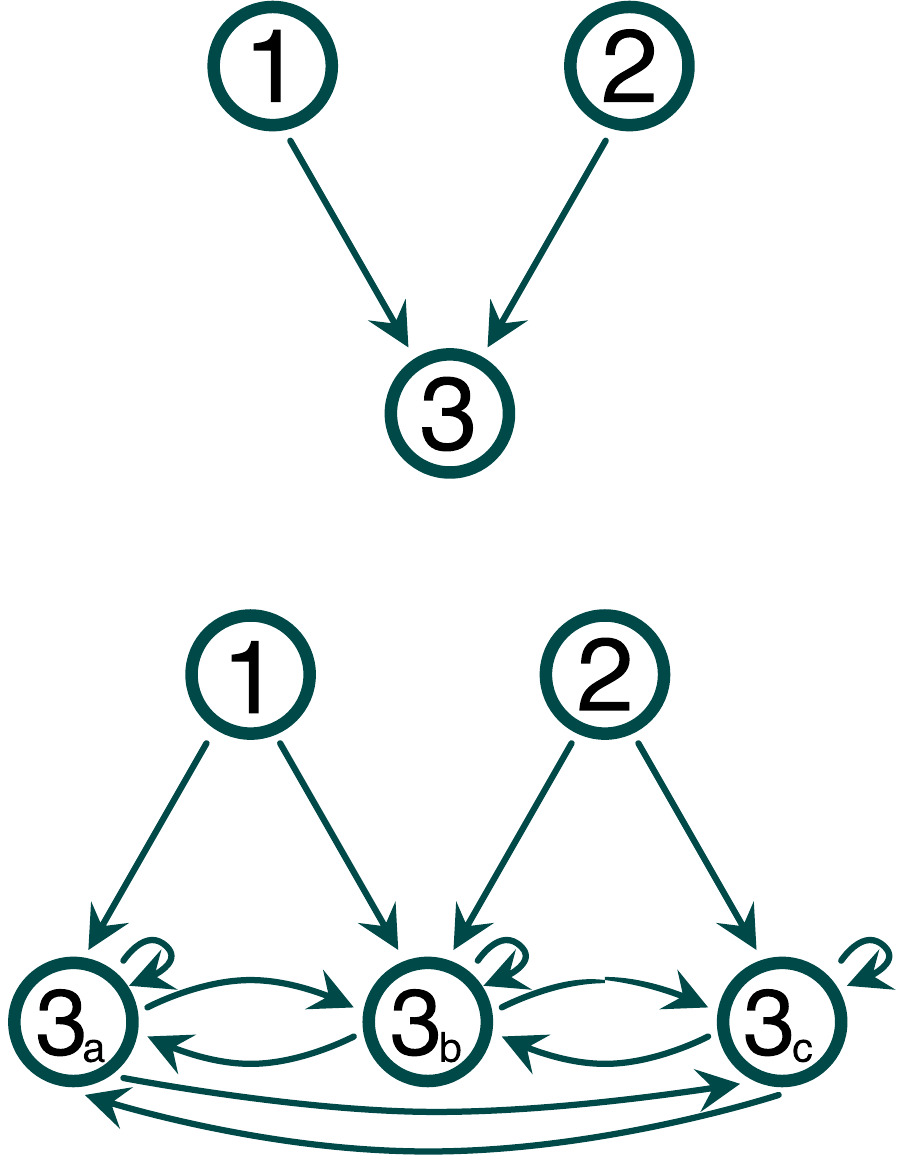}
        
    }
  \caption{
    Left: A simplicial complex $K$ with a multivector field  consisting of two regular multivectors, $V_2$ and $V_5$, 
    and five critical multivectors, $V_1$, $V_3$, $V_4$, $V_6$, $V_7$.
    Top right: the digraph induced by Morse decomposition $\{M_1,M_2,M_3\}$ 
    where $M_1:=V_4$, $M_2:=V_6$, and $M_3=:K\setminus V_4 \setminus V_6$. 
    Bottom right: the digraph induced by Morse predecomposition $\cM$ consisting of: $M_1=V_4$, $M_2=V_6$, $M_3^A=V_1$, $M_3^C=V_3$, and $M_3^E=V_7$.
  }
  \label{fig:mvf-example}
\end{figure}

However, the following example shows that such a minimal Morse predecomposition need not be unique.

\begin{figure}
  \includegraphics[width=0.5\textwidth]{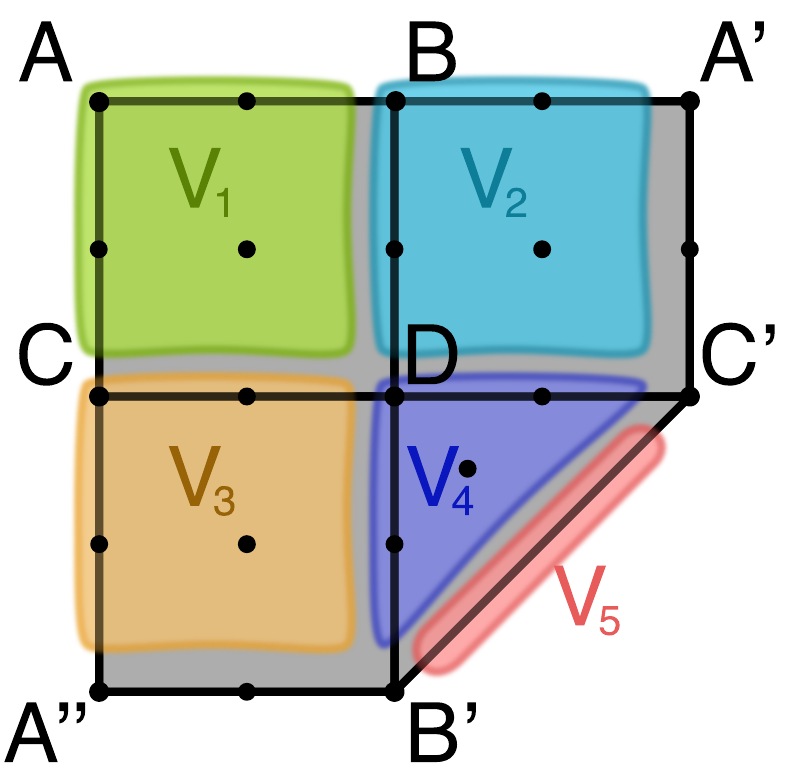}
  \caption{
    A multivector field on a cellular complex in which 
    we assume that edge $AC$ is identified with edge $A'C'$ and edge $AB$ is identified with edge $A'B'$
    The multivector field consists of four regular multivectors, $V_1$, $V_2$, $V_3$, $V_4$, and  critical multivector $V_5$.
  }
  \label{fig:mvf-example2}
\end{figure}

\begin{ex}
\label{ex:mvf-2}
Consider a cellular complex $L$ in Figure \ref{fig:mvf-example2} where
  we assume that edge $AC$ is identified with edge $A'C'$ and edge $AB$ with edge $A'B'$.
Consider a multivector field $\cV$ on $L$ consisting of five multivectors:
\begin{gather*}
  V_1:=\{A,AB,AC,ABCD\},\quad
  V_2:=\{B,A'B,BD,BA'DC'\},\\
  V_3:=\{C,CD,CA'',CDA''B'\},\quad
  V_4:=\{D,DC',DB',DC'B'\},\\
  V_5:=\{B'C'\}.
\end{gather*}
Note that the only critical multivector is $V_5$.
Multivector field $\cV$ on $L$ is recurrent in the sense that for any two cells in $L$ we can construct a path connecting them. 
Thus, the only possible Morse decomposition of $\cV$ is the trivial one, with $L$ as the only Morse set.
However, $\cV$ admits two non-trivial Morse predecompositions.
Let $M_1:=V_1\cup V_2$, $M_2:=V_1\cup V_3$, and $M_3:=V_5$.
Observe that $M_{1}$ is a non-empty isolated invariant set, because for every point in $M_{1}$ we can find a periodic solution, for instance, $\ldots, A, AB, B, BA', A,\ldots$. 
The same observation applies to $M_2$. 
Set $M_3$ is an isolated invariant set as a critical multivector. 
One can verify that $\cM_1:=\{M_1, M_3\}$ and $\cM_2:=\{M_2,M_3\}$ are both non-trivial Morse predecompositions of $S$.
We note that the digraph induced by $\cM_1$ as well as $\cM_2$ is a clique.
Clearly, none of the Morse predecompositions is a refinement of the other.
\exend
\end{ex}

\begin{figure}
\begin{center}
    \includegraphics[width=0.74\textwidth]{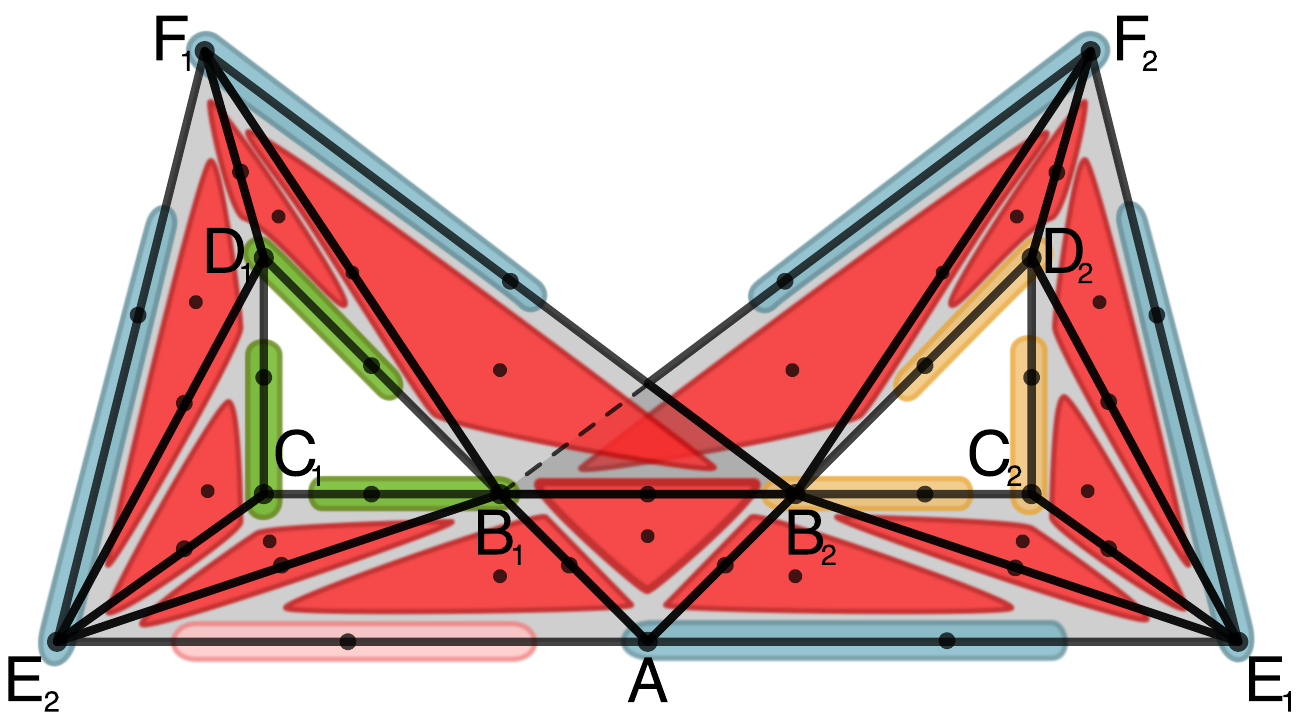}
    \raisebox{5mm}{
        \includegraphics[width=0.23\textwidth]{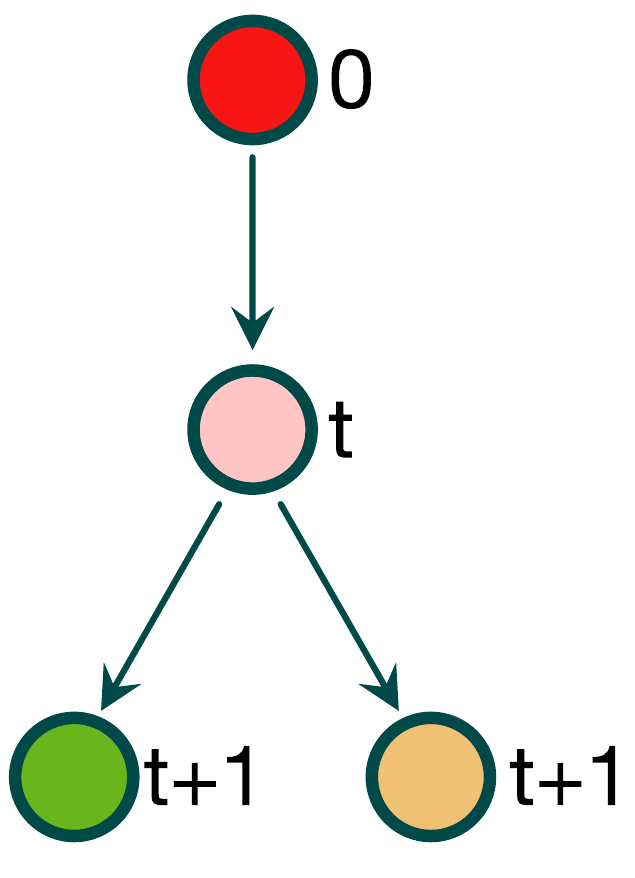}
    }
\end{center}
    \caption{
      Left: A combinatorial model of Lorenz type dynamics with a Morse decomposition  consisting of four Morse sets: a repeller marked with red multivectors, 
      a saddle marked with a pink multivector and two attracting periodic orbits marked respectively with green and yellow multivectors.
      Right: The associated Conley-Morse digraph. 
    }
    \label{fig:lorenz-1}
\end{figure}
The following example presents a combinatorial version of the Williams model~ \cite{Wi1979} of the Lorenz attractor. A similar combinatorial model was introduced in~\cite[Figure 3]{KMW2016} and studied in~ \cite[Figure 4]{MWS2022}.
\begin{ex}
\label{ex:lorenz}
Consider the combinatorial multivector field presented in Figure~\ref{fig:lorenz-1}(left). 
It has exactly one critical multivector, a singleton edge $AE_2$ at bottom left.
All other multivectors are doubletons.
The minimal Morse decomposition consists of four Morse sets: a repeller marked with red multivectors, 
a saddle marked with a pink multivector and two attracting periodic orbits marked respectively with green and yellow multivectors.
The associated Conley-Morse digraph is presented in Figure~\ref{fig:lorenz-1}(right).
In particular, the Conley index of the repeller is zero. 
Such a Conley index says nothing about the invariant set inside. 
However, the repeller contains two repelling periodic orbits sharing the middle bottom triangle $AB_1B_2$. 
Each of them is an isolated invariant set with the cells of the other on a non-trivial homoclinic connection.
It is straightforward to verify that one can obtain a Morse predecomposition 
by replacing the repeller in the Morse decomposition by one of the repelling orbits. 
The resulting two minimal Morse predecompositions together with the Conley-Morse graphs are presented in Figure~\ref{fig:lorenz-2}.
\exend
\end{ex}

\begin{figure}
\begin{center}
    \includegraphics[width=0.74\textwidth]{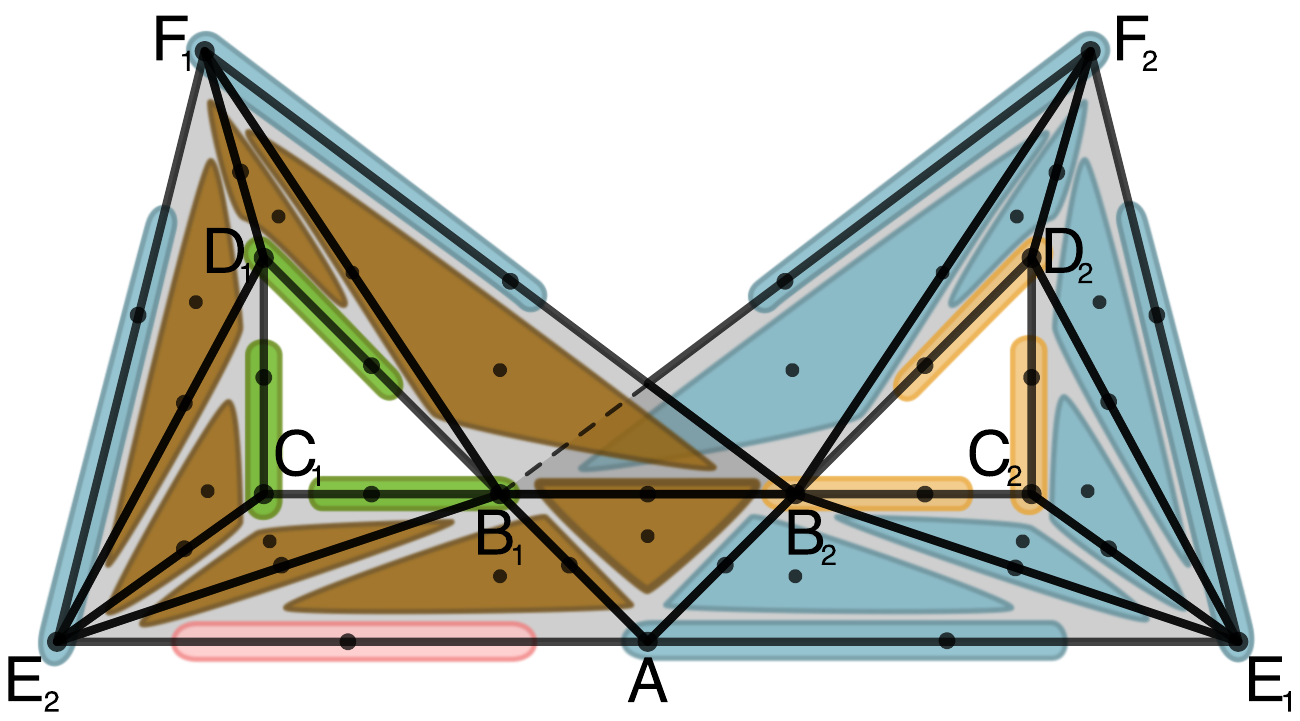}
    \raisebox{5mm}{
        \includegraphics[width=0.23\textwidth]{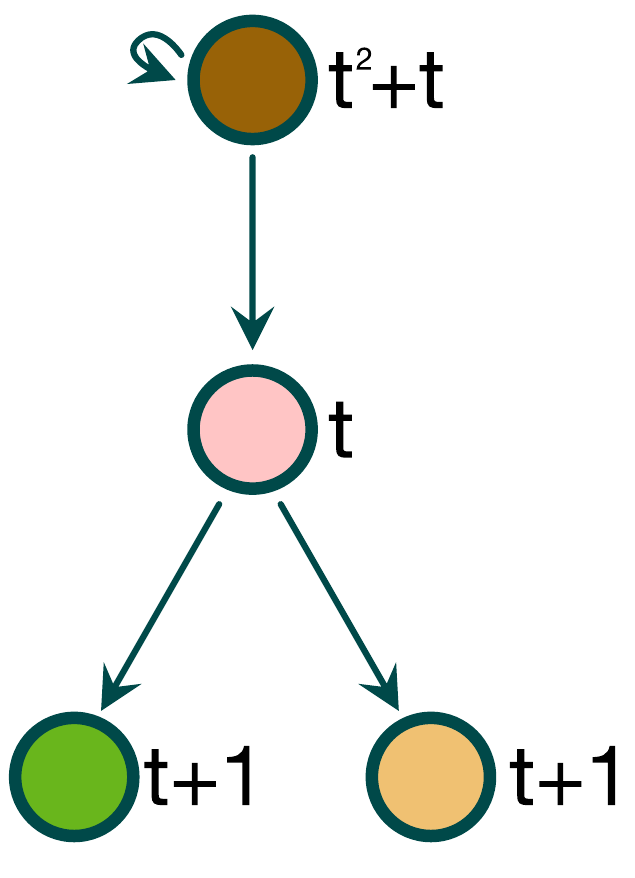}
    }
    
    \includegraphics[width=0.74\textwidth]{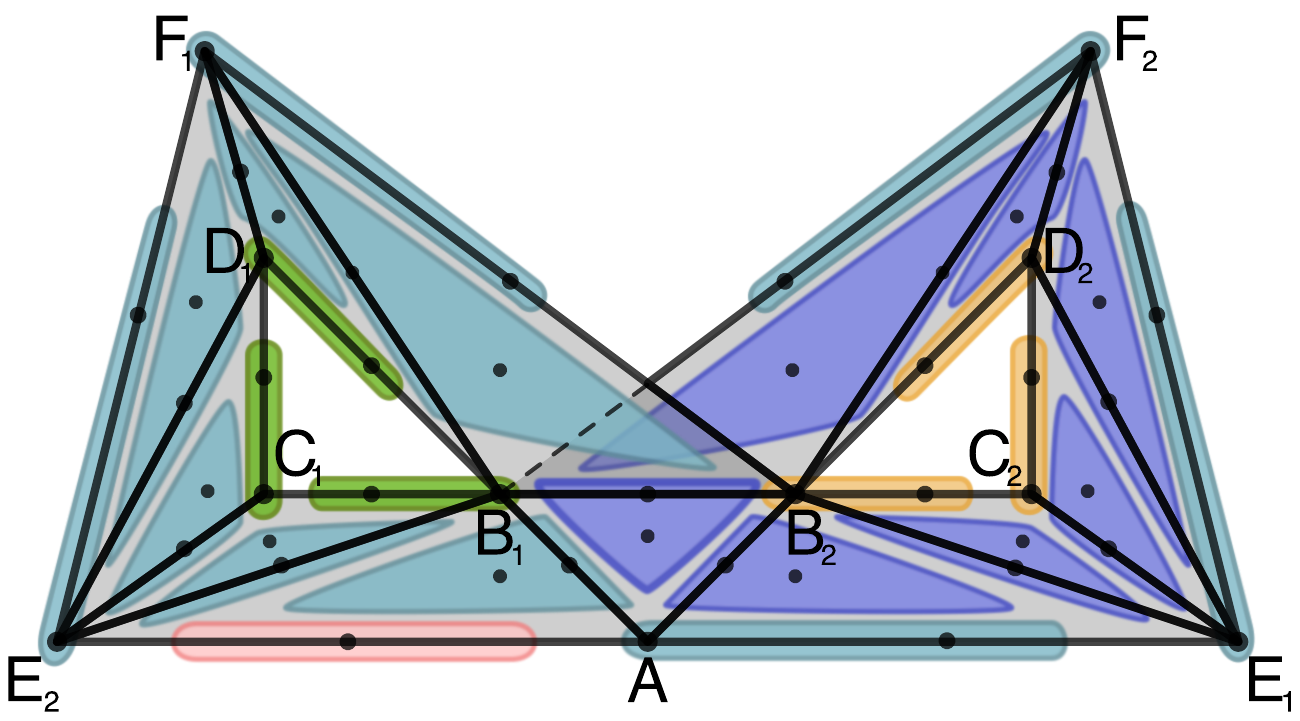}
    \raisebox{5mm}{
        \includegraphics[width=0.23\textwidth]{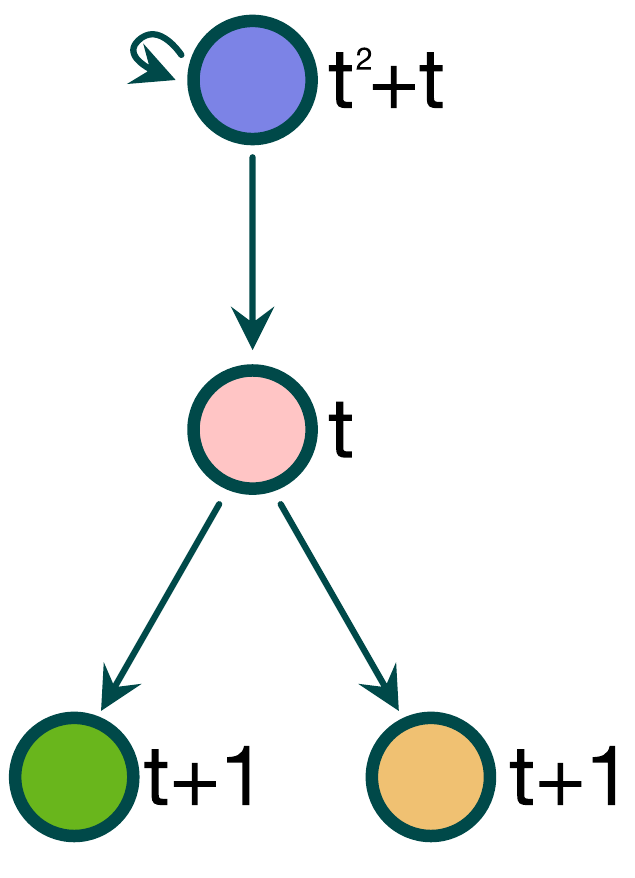}
    }
\end{center}
    \caption{
      Top left: A Morse predecomposition with Conley-Morse digraph consisting of four pre-Morse sets: a repelling periodic orbit on the left marked with brown multivectors, 
      a saddle marked with a pink multivector and two attracting periodic orbits marked respectively with green and yellow multivectors.
      Bottom left: A Morse predecomposition with Conley-Morse digraph symmetric to the top one with the brown repelling orbit replaced by the violet repelling orbit on the right. 
      Top and bottom right: The digraphs induced by the Morse predecompositions on the left.      
    }
    \label{fig:lorenz-2}
\end{figure}

Given a Morse predecomposition  $\cM=\{M_p\mid p\in\VSet\}$ of $X$ and  
$C\subset X$ define
\begin{equation}
    \VSet_C = \VSet_C(\cM):=\setof{p\in\VSet\mid M_p\cap C\neq\emptyset}.
\end{equation}

\begin{lem}
\label{lem:mvf-minimal-element}
Let $\cM$ be a Morse predecomposition of  $X$
and let $C\subset X$ be an isolated invariant subset of $X$ which is strongly connected in $G_\cV$.
If there is an $r\in\PP_C(\cM)$ such that $M_r$ is saturated in $X$ then $C\subset M_r$. 
In particular, $\VSet_C=\{r\}$. 
\end{lem}
\begin{proof}
  To see that $C\subset M_r$ consider an $x \in C$. By Proposition~\ref{prop:strongly-connected-surjection}
  we can construct a path $\rho:[0,n]_\mathbb{Z}\to C$ through $x$ such that $\im\rho=C$ and $\rho(0)=\rho(n)$.
  Since $r\in\VSet_C(\cM)$, we can find a $y\in M_r\cap C$. Then $y=\rho(k)$ for some $k\in [0,n]_\mathbb{Z}$.
  Without loss of generality we may assume that $k=0$. Let $\gamma$ be an essential solution through $y$ in $M_r$
  and let $\gamma':=\gamma^-\cdot\rho\cdot\gamma^+$. Then $\gamma'$ is an essential solution through $x$
  and $\alpha(\gamma')=\alpha(\gamma)\subset M_r$ and $\omega(\gamma')=\omega(\gamma)\subset M_r$.
  Therefore, $\gamma'$ is a homoclinic connection of $M_r$.
  Since $M_r$ is saturated we get $x\in\im\gamma'\subset M_r$.
  Consequently, $C\subset M_r$.
  Since the elements of $\cM$ are mutually disjoint, we also get $\VSet_C=\{r\}$.
  \qed
\end{proof}

The following theorem gives characterizations of  Morse decomposition. 
In particular, condition (iii) shows that Definition~\ref{defn:mvf-Morse-decomposition} is equivalent to Definition~\ref{defn:mvf-Morse-decomposition-orig}
which is the definition of Morse decomposition assumed in~\cite[Definition 7.1]{LKMW2020}.

\begin{thm}
\label{thm:mvf-morse-decomposition}
Let $\cM=\{M_p\mid p\in\VSet\}$ be an indexed  family of  mutually disjoint subsets of $X$.
Then, the following conditions are mutually equivalent.
\begin{itemize}
   \item[(i)] The family $\cM$ is a Morse decomposition of $\cS$.
   \item[(ii)] Each $M_p$ is a saturated, isolated invariant subset of $X$ and 
   $\VSet$ admits a partial order $\leq$ such that for every essential solution $\gamma$ in $X$ there exist $p,q\in\VSet$ 
   satisfying $p\geq q$, $\alpha(\gamma)\subset M_p$ and $\omega(\gamma)\subset M_q$.
   \item[(iii)] Each $M_p$ is an isolated invariant subset of $X$ and 
   $\VSet$ admits a partial order $\leq$ such that for every essential solution $\gamma$ in $X$ 
   either $\im\gamma\subset M_p$ for some $p\in\VSet$ or 
   there exist $p,q\in\VSet$ 
   satisfying $p> q$, $\alpha(\gamma)\subset M_p$ and $\omega(\gamma)\subset M_q$.
\end{itemize}
In particular, every Morse set in a Morse decomposition is saturated and isolated. 
\end{thm}
\proof
Assume (i). Then $\cM$ is a saturated Morse predecomposition and the flow induced preorder $\leq_\cM$ is a partial order. 
In particular, each $M_p$ is a saturated, isolated invariant subset of $X$.
We claim that the partial order $\leq_\cM$ fulfills the requirements in condition (ii). 
To see this consider an essential solution $\gamma$ in $X$.
By the definition of Morse predecomposition there exist $p,q\in\VSet$ such that  $\alpha(\gamma)\cap M_p\neq\emptyset$
    and $\omega(\gamma)\cap M_q\neq\emptyset$. 
By the definition of flow induced preorder we get $p\leq_\cM q$ and by Theorem~\ref{thm:mvf-limit_set_iis} 
    and Lemma~\ref{lem:mvf-minimal-element} we conclude that $\alpha(\gamma)\subset M_p$ and $\omega(\gamma)\subset M_q$.
Hence, (ii) follows from (i).

Assume in turn (ii). Then,  each $M_p$ is an isolated invariant subset of~$X$.
We claim that the partial order $\leq$ given in (ii) fulfills the requirements of condition (iii). 
To see this, consider an essential solution $\gamma$ in $X$ and assume there is no $p\in\VSet$ such that $\im\gamma\subset M_p$.
By (ii) there exist $p,q\in\VSet$ satisfying $p\geq q$, $\alpha(\gamma)\subset M_p$ and $\omega(\gamma)\subset M_q$.
We claim that $p>q$. To see this assume to the contrary that $p=q$.
We first prove that under this assumption we have $\VSet_{\im\gamma}\subset \{p\}$. 
To prove this consider an $r\in \VSet_{\im\gamma}$ such that $r\neq p$.
Then, by Proposition~\ref{prop:ess_path_extension}, one can construct a link from $M_p$ to $M_r$ and a link from $M_r$ to $M_q$, which gives $q\leq r \leq p=q$, which contradicts with the assumption that $\leq$ is a partial order.
It proves that $\VSet_{\im\gamma}\subset \{p\}$.
However, this means that $\gamma$ is a homoclinic connection to $M_p$. 
Since $M_p$ is saturated, we get  $\im\gamma\subset M_p$, a contradiction. 
Hence, $p>q$ which completes the proof of (iii).

Finally, assume (iii) and consider an essential solution $\gamma$ in $X$.
In order to prove that $\cM$ is a Morse predecomposition, we have to show that $\gamma$ is a link from $M_p$ to $M_q$ for some $p,q\in\VSet$.
If there is a $p\in\VSet$ such that $\im\gamma\subset M_p$, then also $\alpha(\gamma)\subset M_p$ and $\omega(\gamma)\subset M_p$,
which shows that $\gamma$ is a link from $M_p$ to $M_p$. 
Otherwise, we get from (iii) that there exist $p,q\in\VSet$ 
satisfying $p> q$, $\alpha(\gamma)\subset M_p$ and $\omega(\gamma)\subset M_q$, 
which also proves that $\gamma$ is a link.  This completes the proof that $\cM$ is a Morse predecomposition.

Clearly, $\leq$ is then an admissible partial order for $\cM$. We still need to prove that each $M_r\in\cM$ is saturated. 
To see this consider a full solution $\gamma$ such that $ \alpha(\gamma)\cup\omega(\gamma)\subset M_r$.
If $\im\gamma\not\subset M_r$, then 
we get from (iii) the existence of $p,q\in\VSet$ satisfying $p> q$, $\alpha(\gamma)\subset M_p$ and $\omega(\gamma)\subset M_q$.
However, Morse sets are mutually disjoint and limit sets are non-empty. Therefore $p=r=q$, a contradiction 
proving that $M_r$ is saturated.
\qed

\subsection{Consolidation theorem}
Assume that $\cV$ is a combinatorial multivector field on a finite topological space $X$ and
consider a fixed Morse predecomposition $\cM=\setof{M_p\mid p\in\VSet}$.
Given $\VSubSet\subset\VSet$ we denote the family of links in $X$ between pre-Morse sets with indexes in $\VSubSet$ by
\begin{equation*}
  \esol_{\VSubSet}(X) := \{\gamma\in\esol(X) \mid
        \VSet_{\alpha(\gamma)}\cap {\VSubSet}\neq\emptyset \neq \VSet_{\omega(\gamma)}\cap \VSubSet\},
\end{equation*}
and we define the \emph{connection set} of $\VSubSet$ by
\begin{equation}\label{eq:connection_set}
  M_{\VSubSet} := \bigcup\{\im\gamma \mid \gamma\in\esol_{\VSubSet}(X) \}.
\end{equation}

\begin{lem}\label{lem:connection_sets}
  Assume $\leq$ is an admissible preorder with respect to Morse predecomposition $\cM$.
  Then for every $\VSubSet\subset\VSet$ the set $M_{\VSubSet}$ is a saturated isolated invariant set.
  Moreover, if $\VSubSet$ is convex with respect to $\leq$ then $M_{\VSubSet}\cap M_r=\emptyset$ for all $r\in\VSet\setminus {\VSubSet}$.
\end{lem}
\begin{proof}
    First observe that $\esol(x,M_{\mathbb{Q}})$ is non-empty for every $x\in M_{\VSubSet}$ directly by definition (\ref{eq:connection_set}).
    Thus, $M_{\VSubSet}$ is invariant.
    We will show that $M_{\VSubSet}$ is $\cV$-compatible and locally closed.
    To prove $\cV$-compatibility take an $x\in M_{\VSubSet}$ and a $\gamma\in\esol(x,M_{\VSubSet})$.
    Consider a $y\in[x]_\cV$ and an essential solution $\gamma':=\gamma^-\cdot y\cdot\gamma^+$.
    Clearly, $\gamma' \in \esol_{\VSubSet}$.
    It follows that $y\in M_{\VSubSet}$.
    Therefore, if $x\in M_{\VSubSet}$ then $\supsetV{x}\subset M_\VSubSet$, which proves the $\cV$-compatibility of $M_{\VSubSet}$. 
    
    To show that $M_\VSubSet$ is locally closed it suffices to prove that $M_\VSubSet$ is convex with respect to the order $\leq_\cT$ on $X$ given by the topology (see \eqref{eq:poset_topology} in Section~\ref{sec:top}).
    Hence, consider $x,z\in M_\VSubSet$ and a $y\in X$ such that $x\leq_\cT y\leq_\cT z$.
    Since there exist $\gamma_x\in\esol(x,M_\VSubSet)$ and $\gamma_z\in\esol(z,M_\VSubSet)$, we can construct an essential solution $\psi:=\gamma_z^-\cdot y\cdot\gamma_x^+$.
    We have $\alpha(\psi)\cap M_p=\alpha(\gamma_z)\cap M_p\neq\emptyset$ and $\omega(\psi)\cap M_q=\omega(\gamma_x)\cap M_q\neq\emptyset$ for some $p,q\in \VSubSet$.
    Hence, $\psi\in\esol_\VSubSet$ and $y\in\im\psi\subset M_\VSubSet$. This proves that $M_\VSubSet$ is convex with respect to $\leq_\VSubSet$, that is locally closed.
    Hence, we get from Theorem~\ref{thm:lcl_vcomp_iis} that $M_\VSubSet$ is an isolated invariant set. It follows directly from the definition of $M_\VSubSet$ that $M_\VSubSet$ is saturated.
    
    To prove the remaining assertion assume that $\VSubSet\subset \VSet$ is convex with respect to the admissible preorder $\leq$ and $M_\VSubSet\cap M_r\neq\emptyset$ for some $r\in\VSet\setminus \VSubSet$.
    Then we can select an $x\in\im\gamma\cap M_r$ for some $\gamma\in\esol(x, M_\VSubSet)$
    and, by Definition \eqref{eq:connection_set}, we can find $p,q\in \VSubSet$ such that $\alpha(\gamma)\cap M_p\neq\emptyset$ and $\omega(\gamma)\cap M_q\neq\emptyset$.
    Since $M_r$ is invariant there exists a $\psi\in\esol(x,M_r)$.
    Consider essential solutions $\gamma_q:=\psi^-\cdot\gamma^+$ and $\gamma_p:=\gamma^-\cdot\psi^+$.
    We have
    \begin{align*}
        \alpha(\gamma_q)\subset M_r,\ 
        \omega(\gamma_q)\cap M_q\neq\emptyset,\  
        \alpha(\gamma_p)\cap M_p\neq\emptyset\text{, and}\
        \omega(\gamma_p)\subset M_r.
    \end{align*}
    By the admissibility of $\leq$  we get $q\leq r\leq p$. Since $p,q\in \VSubSet$ and $r\not\in \VSubSet$,  
    this contradicts the assumption that $\VSubSet$ is convex.
    \qed
\end{proof}

\begin{thm} (Consolidation theorem)
\label{thm:mvf_consolidation}
  Let $\cM=\setof{M_p\mid p\in\VSet}$ be a Morse predecomposition of $X$ and let $\leq$ be an admissible preorder.
  Assume $\VSetBis $ is a partition of $\VSet$ consisting of non-empty, 
  convex with respect to $\leq$ subsets of $\VSet$ such that the induced relation $\leq_\cQ$ is a partial order, 
  where $\leq_\cQ$ is defined for  $Q,Q'\in\VSetBis$ by $Q\leq_\cQ Q'$ if and only if there exist a $q\in Q$ and a $q'\in Q'$ such that $q\leq q'$.
  Then, $\cM' := \left\{M_Q \mid Q\in\VSetBis \right\}$ is a Morse decomposition of $X$.
\end{thm}
\begin{proof}
  We see from Lemma \ref{lem:connection_sets} that   $M_Q$ is a saturated isolated invariant set for every $Q\in\cQ$.
  Consider $M_Q,M_{Q'}\in\cM'$ for some $Q,Q'\in\VSetBis $ such that $Q\neq Q'$.
  To see that $M_Q$ and $M_{Q'}$ are disjoint, suppose to the contrary that there exists an $x\in M_Q\cap M_{Q'}$.
  Select a $\gamma \in\esol(x,M_Q)$ and a $\gamma'\in\esol(x,M_{Q'})$.
  Then $\gamma_{QQ'}:=\gamma^-\cdot(\gamma')^+$ is an essential solution satisfying 
    $\alpha(\gamma_{QQ'})=\alpha(\gamma)\subset M_Q$ and
    $\omega(\gamma_{QQ'})=\omega(\gamma')\subset M_{Q'}$.
    It follows that $\gamma_{QQ'}$ is a link from $M_q$ to $M_{q'}$ for some $q\in Q$ and $q'\in Q'$.
    Since $\leq$ is admissible, we get $q'\leq q$ and in consequence, also  $Q'\leq_\cQ Q$. 
  Similarly, considering $\gamma_{Q'Q}:=\gamma^-\cdot(\gamma')^+$ we get $Q\leq_\cQ Q'$.
  Since $\leq_\cQ$ is a partial order, we get $Q=Q'$, a contradiction. 
  Hence, $\cM'$ is a Morse predecomposition.
 
  Consider now a $\gamma\in\esol(X)$.
  By the definition of Morse predecomposition we can find $q,q'\in\VSet$ 
  such that $\alpha(\gamma)\cap M_q\neq\emptyset$ and $\omega(\gamma)\cap M_{q'}\neq\emptyset$.
  Since $\leq$ is an admissible preorder, we get $q'\leq q$.  
  Since $\VSetBis $ is a partition, there are some $Q,Q'\in\VSetBis $ such that $q\in Q$ and $q'\in Q'$.
  By Proposition \ref{prop:limit_solution} we can find a $\psi\in\esol(X)$ such that $\alpha(\psi)=\omega(\psi)=\im\psi=\alpha(\gamma)$.
  Therefore, $\psi\in\esol_Q$ and $\alpha(\gamma)=\im\psi\subset M_Q$.
  Similarly, we argue that $\omega(\gamma)\subset M_{Q'}$. 
  Since $q'\leq q$ we get $Q'\leq_\cQ Q$. Therefore, $\cM$ satisfies condition (ii) of Theorem~\ref{thm:mvf-morse-decomposition}
  which means that $\cM$ is a Morse decomposition. 
  \qed
\end{proof}

\begin{cor}
    Let $\cM=\{ M_p\mid p\in \VSet\}$ be a Morse predecomposition of $X$.
    If $\cQ$ is a partition of $\VSet$ into strongly connected components of graph $G_\cM$
    then 
    $\cM' := \{M_Q \mid Q\in\cQ\}$
    is a Morse decomposition of $X$.
\end{cor}

\begin{cor}
\label{cor:preMorse=Morse}
  Let $\cM$ be a Morse predecomposition of $X$.
  If  there is a partial order  among admissible preorders in $\VSet$ then
    $\cM' := \{M_{\{p\}} \mid p\in\VSet\}$
    is a Morse decomposition of $X$.
\end{cor}

We note that $M_{\{p\}}$ in general may be bigger than $M_p$, because it may contain points on solutions homoclinic to $M_p$. 
However, $M_{\{p\}}=M_p$ if $M_p$ is saturated. Therefore, we also have the following corollary.

\begin{cor}
  Let $\cM$ be a saturated Morse predecomposition of $X$.
  If  there is a partial order  among admissible preorders in $\VSet$ then
    $\cM$ is a Morse decomposition of $X$.
\end{cor}

\section{Predecompositions and decompositions for flows}\label{sec:predecomposition_continuous}

In this section we introduce the concept of Morse predecomposition for flows.
We show that, as in the combinatorial case, it generalizes Morse decomposition. 
Moreover, we introduce the notion of Conley predecomposition which allows us to formulate a stability result.
However, the analysis of Morse predecompositions for flows becomes more complicated.
In particular, it is an open question if a counterpart of the consolidation theorem (Theorem \ref{thm:mvf_consolidation}) holds for flows.

\subsection{Morse predecompositions}\label{subsec:predecomposition_continuous}
Let $\varphi:X\times\RR\to X$ be a fixed flow on a locally compact metric space $X$
and let $S$ be a compact, non-empty  invariant set with respect to $\varphi$.
Consider closed invariant sets $S_1$, $S_2$ contained in $S$. 

\begin{defn}
\label{defn:links}
A full solution $\gamma:\RR\rightarrow S$
is a {\em link from $S_1$ to $S_2$ in $S$} if $\alpha(\gamma)\cap S_1\neq\emptyset\neq\omega(\gamma)\cap S_2$. 
Link $\gamma$ is {\em homoclinic} if $S_1=S_2$. Otherwise it is {\em heteroclinic}.
Link $\gamma$  is a {\em connection from $S_1$ to $S_2$} if $\alpha(\gamma)\subset S_1$ and $\omega(\gamma)\subset S_2$.
A homoclinic link $\gamma$  is {\em trivial} if $\im\gamma\subset S_1=S_2$. 
Otherwise it is {\em non-trivial}.
Clearly, a trivial link is a homoclinic connection. However, note that not every homoclinic connection is a trivial link. 
\end{defn}

Observe that a link need not be a connection. However, since $S$ is compact, limit sets of full solutions in $S$ are always non-empty
and, in consequence, every connection is a link.

\begin{defn}
\label{defn:predecflow}
By a \emph{Morse predecom\-posi\-tion of} $S$ we mean 
an indexed family $\cM=\{M_p\mid p\in\VSet\}$ of mutually disjoint closed, invariant subsets of $S$ 
such that every full solution  $\gamma:\RR\rightarrow S$ is a link from $M_p$ to $M_q$ for some $p,q\in\VSet$. 
We refer to the sets in $\cM$ as {\em pre-Morse} sets.
\end{defn}

\begin{prop}
    An indexed family $\cM\{M_p\mid p\in\VSet\}$ of  mutually disjoint closed, invariant subsets of $S$
    is a Morse predecom\-posi\-tion of $S$ if and only if 
    we have $T\cap\bigcup\cM\neq\emptyset$ for every non-empty closed invariant set $T\subset S$.
\end{prop}
\begin{proof}
   Assume $\cM=\{M_p\mid p\in\VSet\}$ is a Morse predecom\-posi\-tion of $S$ and $T\subset S$ is a non-empty invariant set.
Select an $x\in T$. Since $T$ is invariant, there is a full solution $\gamma:\RR\to T$ through $x$.
Let $q\in\VSet$ be such that $\omega(\gamma)\cap M_q\neq\emptyset$. Since $\omega(\gamma)\subset\cl\im\gamma\subset\cl T=T$, 
we conclude that $\emptyset\neq \omega(\gamma)\cap M_q \subset T\cap\bigcup\cM$. 

Assume in turn that $T\cap\bigcup\cM\neq\emptyset$ for every non-empty closed invariant set $T\subset S$
and consider a full solution $\gamma:\RR\to S$. 
By Proposition~\ref{prop:limits_are_chain_recurrent} sets $\alpha(\gamma)$ and $\omega(\gamma)$ are closed and invariant.
Hence, we will find some $p,q\in\VSet$ such that $\alpha(\gamma)\cap M_p\neq\emptyset$ and $\omega(\gamma)\cap M_q\neq\emptyset$,
which proves that $\gamma$ is a link and, therefore, $\cM$ is a Morse predecom\-posi\-tion of $S$.
\qed
\end{proof}

The \emph{digraph induced by Morse predecom\-posi\-tion $\cM$} has $\VSet$ as its vertex set 
and an edge from $p$ to $q$ if there exists a non-trivial link from $M_p$ to $M_q$ in $S$.
We denote it $G_\cM$.

We say that a preorder $\leq$ in $\VSet$ is {\em admissible} if the existence of a link from $M_p$ to $M_q$ implies $q\leq p$.
Clearly, the preorder induced by the transpose of $G_\cM$ is admissible.  
We denote it by  $\leq_\cM$ and we call it the {\em flow induced preorder}.

A \emph{Conley model of $\cM$} is a digraph $C=(V,E)$ with a map 
$\nu_{G_\cM, C}:\cM\to V$ such that it preserves the graph induced preorders.
Clearly, $G_\cM$ itself is a Conley model. We call $G_\cM$ the \emph{flow-defined Conley model} of $\cM$.

\begin{ex}
\label{ex:disk1}
Consider the flow in Figure~\ref{fig:fourDisks}(top left).
The equilibrium $M_1$ shown in red is an isolated invariant set.
However, the only Morse predecomposition is 
\[
\cS := \setof{D}.
\]
The flow defined Conley model consist of a single vertex and no edges, $G = \left(\setof{v}, \emptyset \right)$.
Another possible Conley model consists of a single vertex and a self edge, $G = \left(\setof{v}, \setof{(v,v)} \right)$.
\exend
\end{ex}

\begin{figure}
    \includegraphics[width=0.48\textwidth]{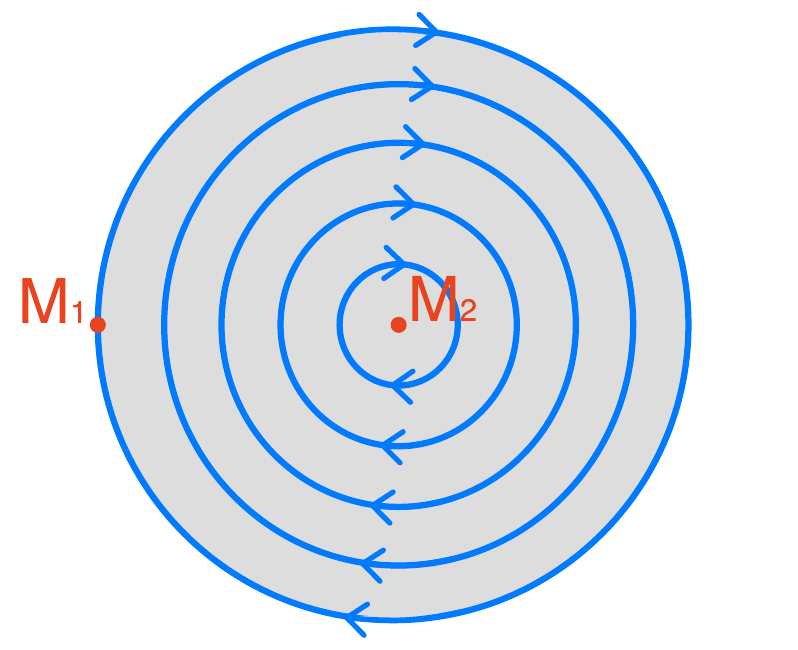}
    \includegraphics[width=0.48\textwidth]{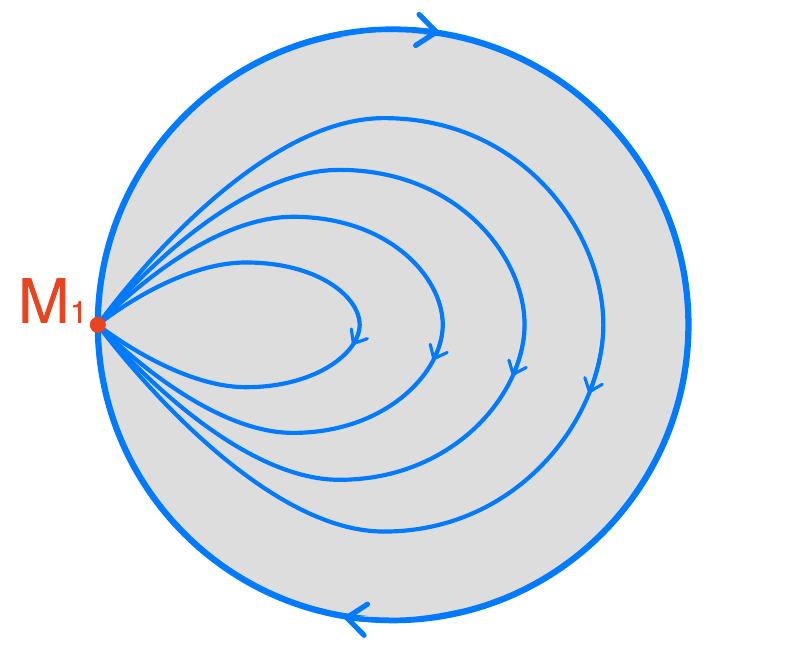}
    
    \includegraphics[width=0.48\textwidth]{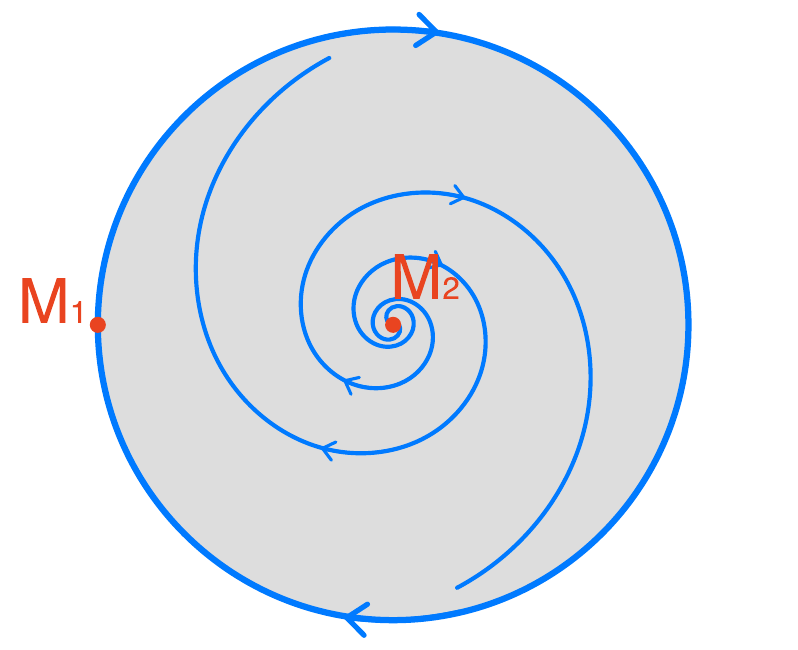}
    \includegraphics[width=0.48\textwidth]{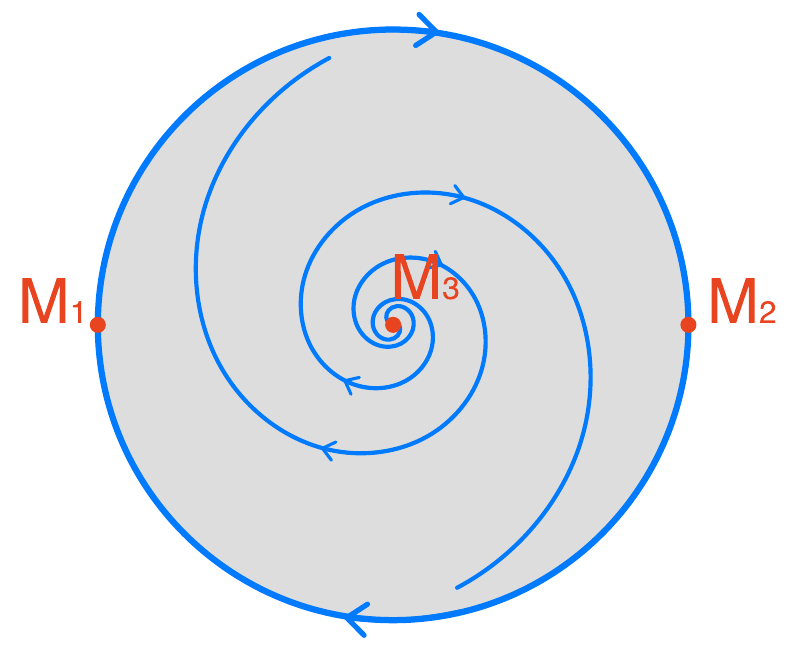}
    \caption{Four examples of flows on the unit disk $D$ in $\RR^2$. Stationary points are indicated in red.
        In all examples $\{D\}$ is a Morse decomposition, a trivial one. 
        Top left: The only  Morse predecomposition is the trivial one $\{D\}$. It is saturated. 
        Top right:  Also $\{M_1\}$ is a  Morse predecomposition. It is not saturated and $M_1$ is not isolated.
        Bottom left:  Also $\{M_1,M_2\}$ is a  Morse predecomposition. It is not saturated but both $M_1$ and $M_2$ are isolated.
        Bottom right: Also $\{M_1,M_2,M_3\}$ is a  Morse predecomposition. It is saturated, hence $M_1$, $M_2$ and $M_3$ are isolated.
    }
    \label{fig:fourDisks}
\end{figure}

\subsection{Saturated invariant sets and Morse predecompositions.}

We say that an invariant subset $T$ of an invariant set $S$ is {\em saturated in $S$} if $\im\gamma\subset T$
for every homoclinic connection $\gamma$ from $T$ to $T$ in $S$.
We say that a Morse predecomposition $\cM$ is {\em saturated} if every $M\in\cM$ is saturated. 

\begin{ex}
\label{ex:disk2}
Consider the flow indicated in Figure~\ref{fig:fourDisks}(top right). It admits Morse predecomposition $\{M_1\}$
consisting of the equilibrium $M_1$ shown in red. Clearly, $M_1$ is not saturated. Therefore, this Morse predecomposition 
is not saturated. Unlike Example~\ref{ex:disk1} the flow defined Conley model consists of a single vertex and necessarily also a loop, 
$G = \left(\setof{v}, \setof{(v,v)} \right)$.
\exend
\end{ex}

\begin{prop}
\label{prop:Morse-sets}
Given a pre-Morse set $M$ in a saturated Morse predecomposition $\cM$,
every compact neighborhood $N$ of $M$ disjoint from all other pre-Morse sets and 
such that $M\subset\inter N$ is an isolating neighborhood isolating $M$. 
In particular, every pre-Morse set in a saturated Morse predecomposition is an isolated invariant set.
\end{prop}
\begin{proof}
    Fix a pre-Morse set $M$ in $\cM$. 
    Let $N$ be a compact neighborhood of $M$ which is disjoint from all other pre-Morse sets in $\cM$.
    We claim that $\Inv N=M$. Clearly, $M\subset \Inv N$. To see the opposite inclusion assume to the contrary that there exists an $x\in\Inv N\setminus M$. 
    Then $x\in S\setminus\bigcup\cM$ and by the definition of Morse predecomposition we can find $M_p,M_q\in\cM$ such that 
    $\alpha(x)\cap M_p\neq\emptyset$, and $\omega(x)\cap M_q\neq\emptyset$.
    Since $x\in\Inv N$, we have $\varphi(x, \RR)\subset N$ and in consequence, $\alpha(x)\cup\omega(x)\subset N$.
    This implies $M_p=M=M_q$, because $N$ is disjoint from pre-Morse sets other than $M$.
    If $M$ is empty then by Proposition \ref{prop:limits_are_chain_recurrent} we get a contradiction.
    Otherwise, since every pre-Morse set in $\cM$ is saturated, we conclude that $x\in M$. 
    This is a contradiction proving that $N$ is an isolating neighborhood isolating $M$.
\qed
\end{proof}

\begin{ex}
\label{ex:disk3}
Consider the flow in Figure~\ref{fig:fourDisks}(bottom left). It admits a Morse predecomposition $\{M_1,M_2\}$
consisting of the equilibriums $M_1$ and $M_2$ shown in red. Here, $M_2$ is saturated but $M_1$ is not saturated. 
Therefore, this Morse predecomposition is not saturated but both pre-Morse sets are isolated invariant sets. 
The flow defined Conley model 
\[
G = \left(\setof{v_1,v_2}, \setof{(v_1,v_1)} \right)
\]
consists of two vertices and two edges. Exactly one of these edges is a loop. 

\exend
\end{ex}

\begin{ex}
\label{ex:disk4}
Consider the flow in Figure~\ref{fig:fourDisks}(bottom right). It admits Morse predecomposition $\{M_1,M_2,M_3\}$
consisting of the equilibriums $M_1$, $M_2$ and $M_3$ shown in red. This Morse predecomposition is saturated.
In particular, all pre-Morse sets are isolated invariant sets. 
The flow defined Conley model 
\[
G = \left(\setof{v_1,v_2,v_3}, \setof{(v_1,v_2),(v_2,v_1),(v_3,v_1),(v_3,v_2)} \right)
\]
consists of three vertices and four edges. None of these edges is a loop. 
\exend
\end{ex}

\begin{ex}
\label{ex:tor_irr}
    Take as the phase space $X$ a torus represented as the quotient of the square $[0,1]\times[0,1]$ in the complex plane through the relation identifying its left edge with the right edge and its bottom edge with the top edge.
    We can identify the equivalence classes with points in $[0,1)\times[0,1)$.
    Consider the flow $\varphi$ on $X$ (see Figure~\ref{fig:torus_irrational_rotation} left) induced by the vector field
    \[
        v:X\ni z\rightarrow k(z)w(z)\in X
    \]
    where $w(z):=1+\theta i$ with $\theta$ an irrational number and
    $k(z):=\min\{|z' - z|\mid z\in\{A,B,C,D\}$ where
    $A=\frac{i}{8}$, $B=\frac{3i}{8}$, $C=\frac{5i}{8}$, and $D=\frac{7i}{8}$.
    Observe that $v$ is a continuous vector field with four stationary points $A$, $B$, $C$, and $D$.
    To understand the features of full solutions of $\varphi$ it is convenient to look into the flow $\bar{\varphi}$
    on $X$ induced by $w$.
    This flow may be considered a suspension of an irrational rotation on a circle. 
    It is well known (for instance see \cite[Theorem 3.2.3]{Silva2008}) 
    that trajectories of an irrational rotation are dense. In consequence, the alpha and omega limit sets
    of a full solution of $\bar{\varphi}$ coincide with the torus.
    Since vector fields $v$ and $w$ differ only by a scalar factor, a full solution $\gamma$ of $\varphi$ is a rescaled 
    full solution of $\bar{\varphi}$. It follows that an alpha or omega limit set of $\gamma$ is either the whole torus $X$
    or one of the four stationary points. 
    Therefore, the only Morse decomposition is the trivial one, consisting of a single Morse set $X$; while the family
        $\cM:=\setof{\{A\},\{B\},\{C\},\{D\}}$ is a Morse predecomposition.
    However, it follows that we can find a link between every pair of sets in $\cM$.
    Thus, the digraph induced by $\cM$ is a clique (see Figure~\ref{fig:torus_irrational_rotation} right).
    Moreover, notice that $\cM$ is saturated.
\exend
\end{ex}

\begin{figure}
  \includegraphics[width=0.50\textwidth]{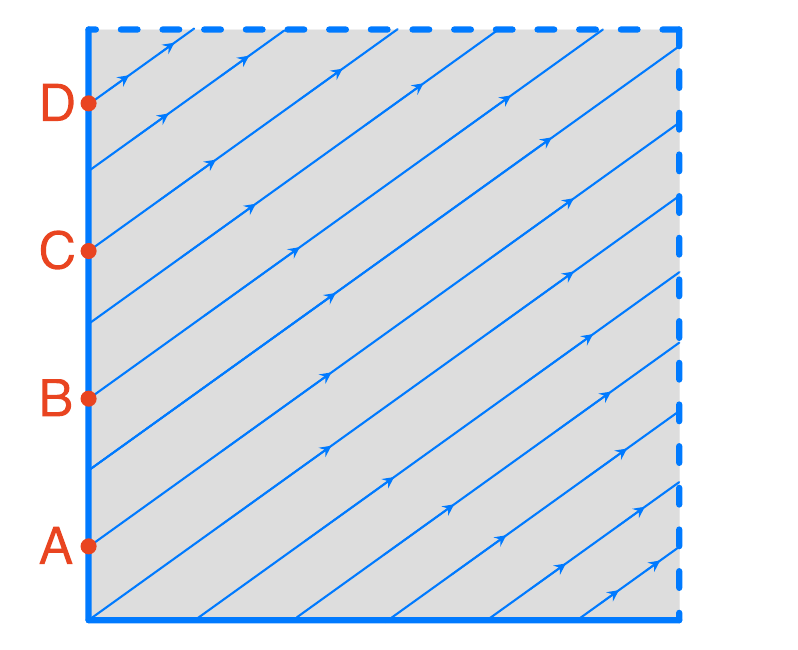}
  \raisebox{0.15\height}{\includegraphics[width=0.30\textwidth]{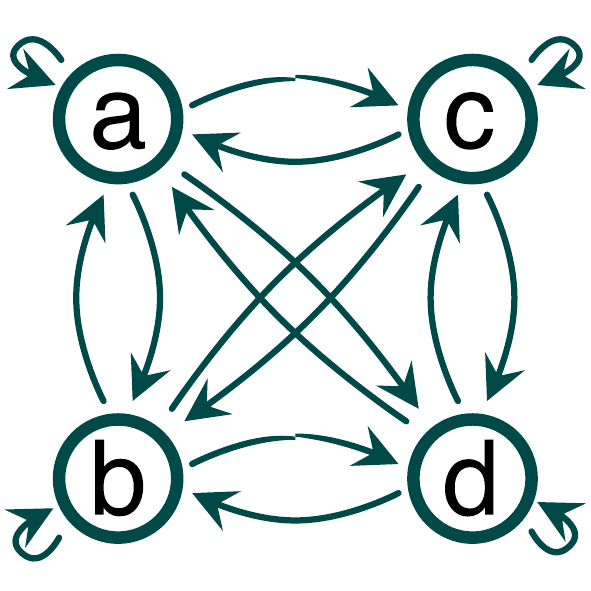}}
  \caption{Left: an irrational rotation on a torus slowing down around four stationary points $A$, $B$, $C$, and $D$.
  Note that there are no heteroclinic nor homoclinic connections between the points. 
  The torus has a Morse predecomposition $\cM:=\setof{\{A\},\{B\},\{C\},\{D\}}$.
  There is a link between any two of these points.
    Right: flow-defined Conley model of $\cM$.
  }
  \label{fig:torus_irrational_rotation}
\end{figure}

Given a Morse predecomposition $\cM=\{M_p\mid p\in\VSet\}$ of $S$ and  
$C\subset S$ define
\[
  \VSet_C := \VSet_C(\cM):=\{p\in\VSet\ |\ C\cap M_p\neq\emptyset\}.
\]

\begin{lem}
\label{lem:minimal-element}
Let $\cM$ be a Morse predecomposition of an invariant set $S$
and let $C\subset S$ be a closed, connected and chain recurrent subset of $S$.
If $\VSet_C$ admits a minimal element $r$ with respect to an admissible preorder $\leq$ and  $M_r$ is saturated in $S$,
then $C\subset M_r$. In particular, $\VSet_C=\{r\}$. Moreover, the conclusion follows if $\cM$ allows for an admissible partial order,
because every finite set has a minimal element with respect to a poset. 
\end{lem}
\begin{proof}
  Assume to the contrary that the inclusion $C\subset M_r$ does not hold. 
  Then $A:=C\cap M_r\varsubsetneq C$. 
  We claim that $A$ is an isolated invariant subset of $C$. 
  Clearly, $A=C\cap M_r$ is invariant as an intersection of invariant sets. 
  To see that $A$ is isolated in $C$ take $N$, an isolating neighborhood for $M_r$. 
  We have 
  \[
    \Inv(C\cap N)\subset\Inv C\cap \Inv N=C\cap M_r=A
  \] 
  and
  \[
    A=C\cap M_r\subset C\cap \Int N\subset \Int_C(C\cap N),
  \]
  which shows that $C\cap N$ is an isolating neighborhood for $A$ in $C$. 
  This proves that $A$ is an isolated invariant set in $C$. 
  By Theorem~\ref{thm:exist_alpha_omega_point} we can find an $x\in C\setminus A=C\setminus M_r$ such that $\alpha(x)\subset A\subset M_r$. 
  We get $\omega(x)\subset C$, because $C$ is invariant and closed. 
  By the definition of Morse predecomposition there exist $p,q\in\VSet$, such that $\alpha(x)\cap M_p\neq\emptyset$
    and $\omega(x)\cap M_q\neq\emptyset$. 
  Since $\alpha(x)\subset M_r$ and pre-Morse sets are mutually disjoint, we must have $p=r$.
  We also get  $C\cap M_q\supset \omega(x)\cap M_q\neq\emptyset$ which implies $q\in\VSet_C$. 
  Since $q\leq p=r$ and $r$ is minimal in $\VSet$, we get $q=r=p$.  
  Hence, $x\in M_r$, because $M_r$ is saturated, a contradiction. 
  Thus, $C\subset M_r$ is proved. 
  Since the elements of $\cM$ are mutually disjoint, we also get $\VSet_C=\{r\}$.
  \qed
\end{proof}

\begin{prop}
\label{prop:saturated-M-pre}
Let $\cM$ be a Morse predecomposition.
$\cM$ is saturated if and only if for every non-trivial link $\gamma$ 
there are $p,q\in\VSet$ such that $p\neq q$ and $\gamma$ is a link from $M_p$ to $M_q$.
\end{prop}
\proof
Assume $\cM$ is a saturated Morse predecomposition and $\gamma$ is a non-trivial link.
Then, by the definition of Morse predecomposition we know that $\VSet_{\alpha(\gamma)}\neq\emptyset\neq\VSet_{\omega(\gamma)}$.
We claim that there is no $r\in\VSet$ such that $\VSet_{\alpha(\gamma)}=\VSet_{\omega(\gamma)}=\{r\}$.
Indeed, if such an $r\in\VSet$ exists, then by Lemma~\ref{lem:minimal-element} we get 
    $\alpha(\gamma)\subset M_r$ and $\omega(\gamma)\subset M_r$.
Since $M_r$ is saturated we conclude that $\im\gamma\subset M_r$, which contradicts the non-triviality of $\gamma$. 
Therefore, 
$\VSet_{\alpha(\gamma)}\cup\VSet_{\omega(\gamma)}$ contains at least two different elements, among which we can find $p$ and $q$ such that $\gamma$ is a link from a $p$ to a $q$ and $p\neq q$. 

To prove the opposite implication suppose that $\cM$ is not saturated.
It means that we can find $r\in\VSet$ with a full solution $\gamma$ such that $\alpha(\gamma)\cup\omega(\gamma)\subset M_r$ and $\im\gamma\not\subset M_r$. 
It follows that $\VSet_{\alpha(\gamma)}\cup\VSet_{\omega(\gamma)}=\{r\}$.
On the other hand, assumptions implies that $\gamma$ is a link from $M_p$ to $M_q$ for some $p\neq q$ in $\VSet$.
Therefore, $\VSet_{\alpha(\gamma)}\cup\VSet_{\omega(\gamma)}$ contains at least two elements, which gives a contradiction.
\qed

\subsection{Morse decompositions}
We now present a definition of Morse decomposition for flows formulated in terms of Morse predecomposition.
However, by Theorem~\ref{thm:morse-decomposition} below, this definition is equivalent to the classical definition (Definition~\ref{defn:flow-Morse-decomposition-orig}).
\begin{defn}
A Morse predecomposition  $\cM=\{M_p\mid p\in\VSet\}$ is called a {\em Morse decomposition} if it is saturated
and among admissible preorders there is a partial order. 
Then, by Proposition~\ref{prop:po-extension}, the flow induced preorder is a partial order. 
\end{defn}

\begin{thm}
\label{thm:morse-decomposition}
Let $\cM=\{M_p\mid p\in\VSet\}$ be an indexed family of  mutually disjoint subsets of $S$.
Then, the following conditions are mutually equivalent.
\begin{itemize}
   \item[(i)] The family $\cM$ is a Morse decomposition of $\cS$.
   \item[(ii)] Each $M_p$ is a closed, saturated invariant subset of $S$ and 
   $\VSet$ admits a partial order $\leq$ such that for every full solution $\gamma$ in $S$ there exist $p,q\in\VSet$ 
   satisfying $p\geq q$, $\alpha(\gamma)\subset M_p$ and $\omega(\gamma)\subset M_q$.
   \item[(iii)] Each $M_p$ is a closed, isolated invariant subset of $S$ and 
   $\VSet$ admits a partial order $\leq$ such that for every full solution $\gamma$ in $S$ 
   through an $x\not\in\bigcup\setof{M_p\mid p\in\VSet}$
   there exist $p,q\in\VSet$ 
   satisfying $p> q$, $\alpha(\gamma)\subset M_p$ and $\omega(\gamma)\subset M_q$.
\end{itemize}
In particular, every Morse set in a Morse decomposition is saturated and isolated. 
\end{thm}
\proof
Assume (i). Then $\cM$ is a saturated Morse predecomposition and the flow induced preorder $\leq_\cM$ is a partial order. 
In particular, each $M_p$ is closed, invariant and saturated. 
We claim that the partial order $\leq_\cM$ fulfills the requirements in condition (ii). 
To see this consider a full solution $\gamma$ in $S$.
By the definition of Morse predecomposition there exist $p,q\in\VSet$ such that  $\alpha(\gamma)\cap M_p\neq\emptyset$
and $\omega(\gamma)\cap M_q\neq\emptyset$. By the definition of flow induced preorder we get $p\leq_\cM q$
and from Lemma~\ref{lem:minimal-element} we conclude that $\alpha(\gamma)\subset M_p$ and $\omega(\gamma)\subset M_q$.
Hence, (ii) follows from (i).

Assume in turn (ii). Then, obviously, each $M_p$ is a closed, invariant subset of $S$.
By Proposition~\ref{prop:Morse-sets} each $M_p$ is also isolated.
We claim that the partial order $\leq$ given in (ii) fulfills the requirements in condition (iii). 
To see this consider a full solution $\gamma$ in $S$ through an $x\not\in\bigcup\setof{M_p\mid p\in\VSet}$.
By (ii) there exist $p,q\in\VSet$ satisfying $p\geq q$, $\alpha(\gamma)\subset M_p$ and $\omega(\gamma)\subset M_q$.
We cannot have $p=q$, because then $M_p=M_q$ is not saturated, contradicting (ii). Therefore, $p>q$, proving (iii).

Finally, assume (iii) and consider a full solution $\gamma$ in $S$.
In order to prove that $\cM$ is a Morse predecomposition, we have to show that $\gamma$ is a link from $M_p$ to $M_q$ for some $p,q\in\VSet$.
Consider first the case when $\im\gamma\subset\bigcup\setof{M_r\mid r\in\VSet}$.
Since $\im\gamma$ is connected and sets $M_r$ are compact and mutually disjoint, we must have $\im\gamma\subset M_r$ for some $r\in\VSet$.
Then $\alpha(\gamma)\cup\omega(\gamma)\subset\cl\im\gamma\subset M_r$. Therefore, $\gamma$ is a link from $M_r$ to $M_r$.
Consider now the case when $\im\gamma\not\subset\bigcup\setof{M_r\mid r\in\VSet}$.
Then there is an $x\in\im\gamma\setminus\bigcup\setof{M_r\mid r\in\VSet}$ and we get from (iii)
the existence of $p,q\in\VSet$ satisfying $p> q$, $\alpha(\gamma)\subset M_p$ and $\omega(\gamma)\subset M_q$.
In particular, $\gamma$ is a link from $M_p$ to $M_q$. This completes the proof that $\cM$ is a Morse predecomposition.
Clearly, $\leq$ is then an admissible partial order for $\cM$. We still need to prove that each $M_r\in\cM$ is saturated. 
To see this consider a full solution $\gamma$ such that $ \alpha(\gamma)\cup\omega(\gamma)\subset M_r$.
If $\im\gamma\not\subset M_r$, then also $\im\gamma\not\subset \bigcup\setof{M_r\mid r\in\VSet}$ and we get from (iii)
the existence of $p,q\in\VSet$ satisfying $p> q$, $\alpha(\gamma)\subset M_p$ and $\omega(\gamma)\subset M_q$.
However, Morse sets are mutually disjoint and limit sets are non-empty. Therefore $p=r=q$, a contradiction 
proving that $M_r$ is saturated.
\qed

\begin{thm}
\label{thm:saturated}
Assume $\cM=\{M_p\mid p\in\VSet\}$ is a saturated Morse predecomposition of $S$.
If there is a partial order among the admissible preorders, then $\cM$ is a Morse decomposition.
\end{thm}
\begin{proof}
    Assume $\leq$ is an admissible partial order for $\cM$. 
    Consider a non-trivial link $\gamma$.
    By Proposition~\ref{prop:saturated-M-pre} there are $p,q\in\VSet$ such that 
        $q<p$ and $\alpha(\gamma)\cap M_p\neq\emptyset\neq\omega(\gamma)\cap M_q$.
    By Proposition~\ref{prop:limits_are_chain_recurrent} the limit sets
        $\alpha(\gamma)$ and $\omega(\gamma)$ are invariant, chain recurrent and connected. 
    It follows that $\VSet_{\alpha(\gamma)}\neq\emptyset\neq\VSet_{\omega(\gamma)}$.
    Therefore, both these sets admit minimal elements.
    From Lemma~\ref{lem:minimal-element} we get $\alpha(\gamma)\subset M_p$ and $\omega(\gamma)\subset M_q$.
    This proves that $\cM$ is actually a Morse decomposition.
    \qed
\end{proof}

\subsection{Conley predecompositions}

A set $A\subset\RR$ is \emph{right infinite} if 
\[
    \sup\setof{b-a\mid a,b\in\RR^+,\;[a,b]\subset A} = \infty.
\]
Similarly,  $A\subset\RR$ is \emph{left infinite} if 
\[
    \sup\setof{b-a\mid a,b\in\RR^-,\;[a,b]\subset A} = \infty.
\]
We say that $A\subset\RR$ is \emph{bi-infinite} if it is both left and right infinite.

\begin{defn}

A finite collection $\cN=\{N_v\mid v\in \VSetV\}$ of mutually disjoint isolating neighborhoods in $X$ 
is called a \emph{Conley predecom\-posi\-tion of} $S$ 
if for every $x\in X$, $\setof{t\in\RR\mid \varphi(t,x)\in \bigcup\{\inter N\mid N\in\cN \}}$ is bi-infinite.
\end{defn}

The \emph{associated digraph} $G_\cN$ has $\VSetV$ as its set of vertices and an edge from $v$ to $w$ if there is a partial solution $\rho:[a,b]\rightarrow S$ satisfying $\gamma(a)\in\inter N_v$ and $\gamma(b)\in\inter N_w$.

A Conley predecom\-posi\-tion $\cN=\{N_v\mid v\in \VSetV\}$ is \emph{inscribed} in Conley predecom\-posi\-tion $\cN'=\{N'_w\mid w\in \WW\}$
if for every $v\in \VSetV$ there exists a $\nu(v):=w\in W$ such that $N_v\subset N'_w$.
Then, we have a well defined map preserving the digraph induced preorders
$\nu:\VSetV\rightarrow \WW$.

We say that a Conley predecom\-posi\-tion $\cN$ is an \emph{approximation} of a Morse predecom\-posi\-tion $\cM=\{M_p\mid p\in\VSet\}$ if for every $p\in\VSet$ there is a $\mu(p):=v\in \VSetV$ such that $\inv N_v=M_p$.
Clearly, such a $v$ is uniquely determined by $p$, because if $M_p=\inv N_{v_i}$ for $i=1,2$ then $\emptyset\neq M_p=\inv N_{v_i}\subset \inter N_{v_i}$, which gives $v_1=v_2$.
Therefore, we have a well defined map $\mu:\VSet\rightarrow\VSetV$.
\begin{prop}
    Let $\cN=\{N_v\mid v\in\VSetV\}$ be a Conley predecom\-posi\-tion approximating Morse predecom\-posi\-tion $\cM=\{M_p\mid p\in\VSet\}$.
    Then $G_\cN$ is a Conley model for $\cM$ with map $\mu:\VSet\rightarrow \VSetV$ preserving the digraph induced orders given by $\mu(p):=v$, where $\inv N_v=M_p$.
\end{prop}
\begin{proof}
   To see that $\mu$ is injective, observe that $\mu(p)=\mu(q)$ for some $p,q\in\VSet$
   implies $M_p=\Inv N_{\mu(p)}=\Inv N_{\mu(q)}=M_q$ which gives $p=q$.
   If $(p,q)$ is an edge in $G_\cM$, then there is a full solution $\gamma:\RR\to S$ such that $\alpha(\gamma)\cap M_p\neq\emptyset$
   and $\omega(\gamma)\cap M_q\neq\emptyset$. Hence, we can find $t,s\in\RR$, $t<s$ such that $\gamma(t)\in\inte N_{\mu(p)}$ and $\gamma(s)\in\inte N_{\mu(q)}$.
   This shows that there is also an edge in $G_\cN$ from $\mu(p)$ to $\mu(q)$, which proves that $\mu$ induces a relation preserving map.
\qed
\end{proof}

\begin{figure}
  \includegraphics[width=0.75\textwidth]{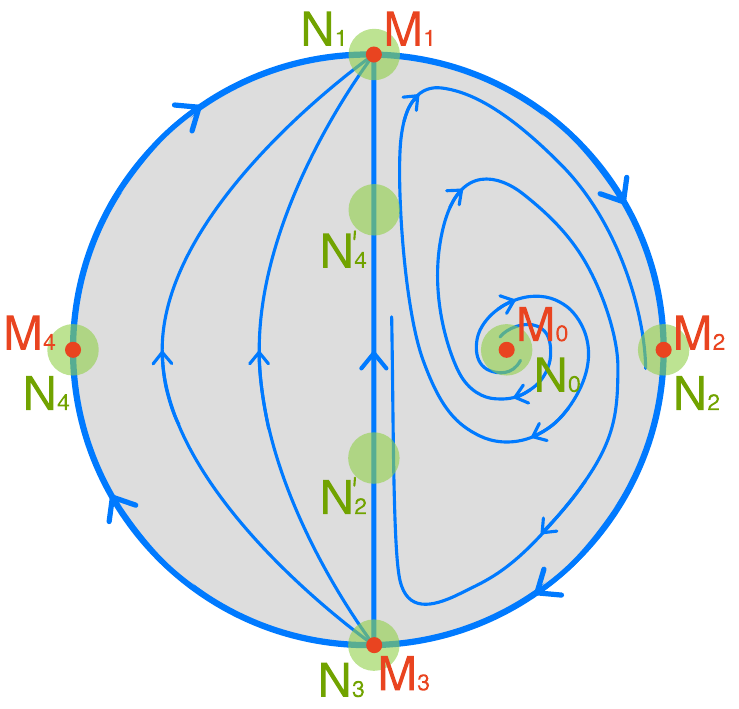}
  
  \includegraphics[width=0.45\textwidth]{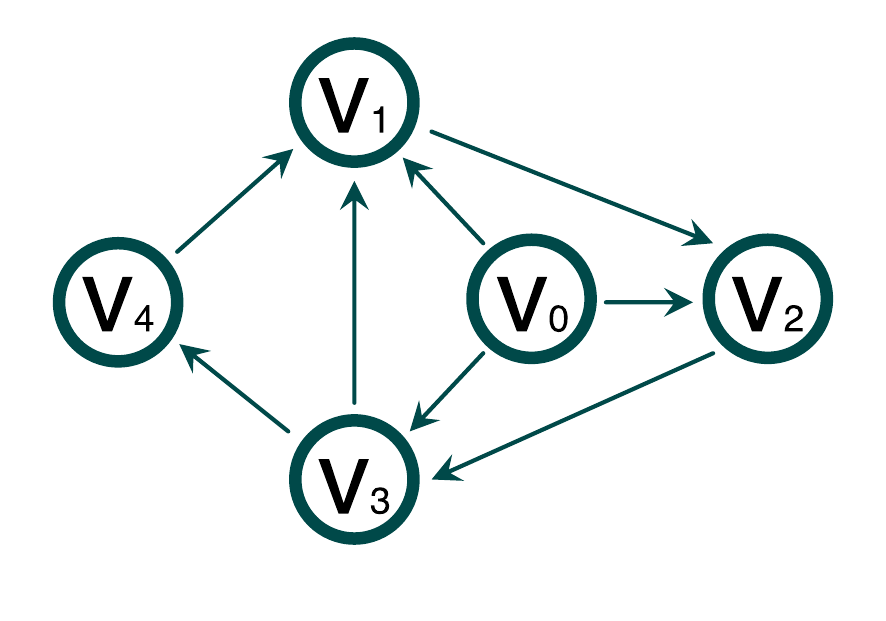}
  \includegraphics[width=0.45\textwidth]{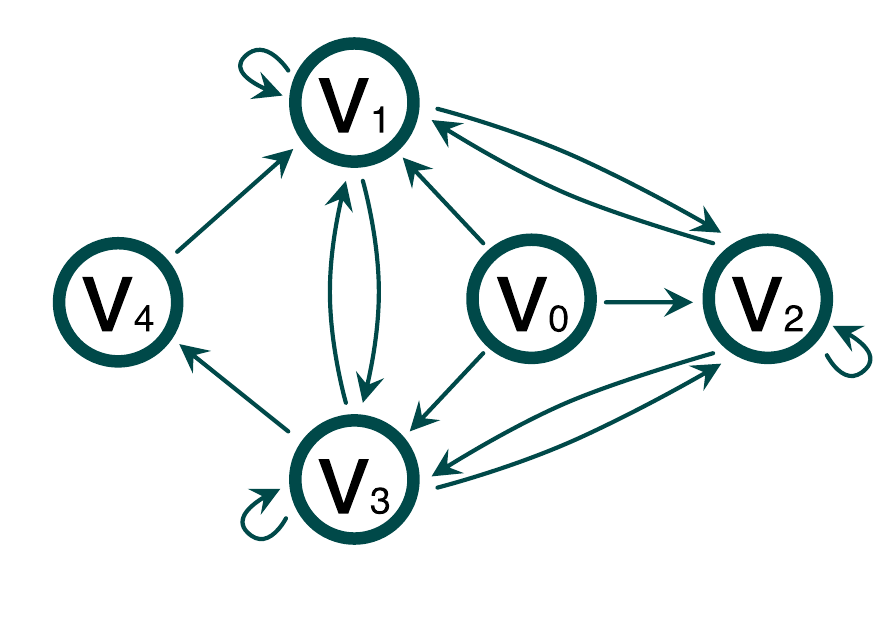}
  
  \includegraphics[width=0.45\textwidth]{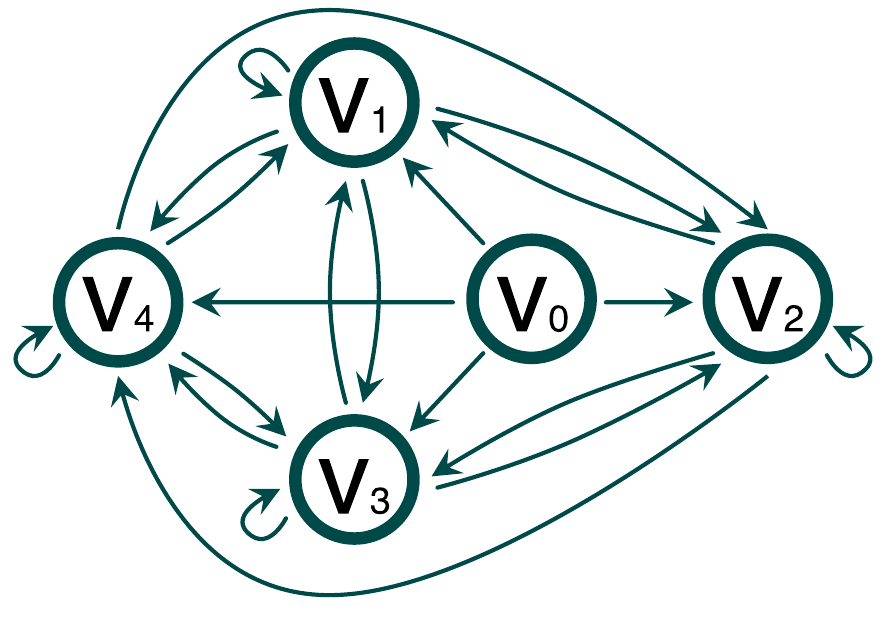}
  
  \caption{Top: An isolated invariant set $S$ consisting of a planar disk 
      with a minimal Morse predecomposition $\cM=\{M_0,M_1,M_2,M_3,M_4\}$ comprising five stationary points  marked in red,
      and Conley predecompositions $\cN:=\{N_0,N_1,N_2,N_3,N_4\}$ 
      and 
      $\cN':=\{N_0,N_1,N_2\cup N'_2,N_3,N_4\cup N'_4\}$.
      Middle left: The flow defined Conley model for $\cM$.
      Middle right: The  Conley model derived from $\cN$.
      Bottom: The  Conley model derived from $\cN'$.
  }
  \label{fig:M0-4}
\end{figure}

\begin{ex}
\label{ex:M0-4}
Consider the flow presented in Figure~\ref{fig:M0-4}(top) and 
the isolated invariant set $S$ consisting of a planar disk. It admits a minimal Morse predecomposition $\cM:=\{M_0,M_1,M_2,M_3,M_4\}$
where each pre-Morse set is a stationary point, marked in red.
    The flow induced Conley model is presented in middle left.
Taking a small disk around each stationary point we obtain a Conley decomposition $\cN:=\{N_0,N_1,N_2,N_3,N_4\}$ approximating $\cM$.
    The associated digraph $G_\cN$ is presented in Figure~\ref{fig:M0-4}(middle right). 
We note that not every Conley decomposition approximating $\cM$ will give the same digraph.
For instance, for $\cN':=\{N_0,N_1,N_2\cup N'_2,N_3,N_4\cup N'_4\}$
we obtain another Conley decomposition approximating $\cM$ with the associated Conley model presented in Figure~\ref{fig:M0-4}(bottom). 
\exend
\end{ex}

\begin{prop}
Let $\cN$ be a Conley predecomposi\-tion of $S$.
Then 
\[
    \cN^\bullet :=\cN^\bullet_\varphi :=\setof{\Inv N\mid N\in\cN,\, \Inv(N, \varphi)\neq\emptyset}
\]
is a Morse predecom\-posi\-tion of $S$ and $\cN$ is its approximation.
\end{prop}
\begin{proof}
    Let $\rho:\RR\to S$ be full solution. 
    Since $\rho^{-1}\left(\bigcup\{\inter N\mid N\in\cN \}\right)$ is bi-infinite and $\cN$ is finite, there are $N_-,N_+\in\cN$ such that $\rho^{-1}(N_-)$ is left infinite and $\rho^{-1}(N_+)$ is right infinite.
    We claim that $\alpha(\rho)\cap \Inv N_-\neq\emptyset$ and $\omega(\rho)\cap \Inv N_+\neq\emptyset$.
    Since $\rho^{-1}(N_+)$ is right infinite, we can find a sequence $t_n$ of real numbers such that $\rho([t_n-n,t_n+n])\subset N_+$ and $\lim_{n\to\infty}t_n=\infty$ for $n\in\NN$. 
    Let $x_n:=\rho(t_n)\in N_+$. 
    Without loss of generality we may assume that $\lim x_n=x_*\in N_+$.
    Then, there is a full solution $\gamma:\RR\to N_+$ through $x$. 
    This shows that $x_*\in\Inv N_+\in \cN^\bullet$. 
    Clearly, also $x_*\in\omega(\rho)$.
    Therefore, $\omega(\rho)\cap\Inv N_+\neq\emptyset$.
    Similarly we prove that $\alpha(\rho)\cap\Inv N_-\neq\emptyset$.
    This proves that  $\cN^\bullet$ is a Morse predecomposition and $\cN$ is its approximation.
\qed
\end{proof}

Consider a continuously parametrized family of dynamical systems
\[
    \varphi_\lambda:\RR\times X\rightarrow X,\quad \lambda\in[-1,1]
\]
on a compact metric space $X$.

\begin{thm}(Stability of Conley predecom\-posi\-tion)
    Let $\cN$ be a Conley predecom\-posi\-tion of $X$ with respect to $\varphi_0$.
    Then, there exists an $\varepsilon>0$ such that $\cN$ is an approximation for $\cN^\bullet_{\varphi_{\lambda}}$ for every $\lambda\in(-\varepsilon,\varepsilon)$. 
    In particular, $G_\cN$ is a Conley model of Morse predecomposition $\cN^\bullet_{\varphi_\lambda}$ for $\lambda\in(-\varepsilon,\varepsilon)$.
\end{thm}
\begin{proof}
    Since $\cN$ consists of a finite number of isolating neighborhoods, the conclusion follows immediately
from Proposition~\ref{prop:perturbation}.
\qed
\end{proof}

\subsection{Asymptotic characterization of Morse predecompositions}
In this section we study Morse predecompositions whose pre-Morse sets are isolated invariant sets.
In such a case we provide theorems which characterize Morse predecomposition in terms of the asymptotic behavior
of trajectories in isolating neighborhoods surrounding the isolated invariant sets. 

\begin{prop}
\label{prop:inifinite-capacity}
A subset $A\subset\RR$ is right infinite if and only if
\begin{equation}
\label{eq:inifinite-capacity-plus}
\forall K>0\;\exists a\geq K\;\; [a,a+K]\subset A.
\end{equation}
A subset $A\subset\RR$ is left infinite if and only if
\begin{equation*}
\forall K>0\;\exists a\leq -K\;\; [a-K,a]\subset A.
\end{equation*}
\end{prop}
\proof
    Obviously condition \eqref{eq:inifinite-capacity-plus} implies that $A$ is right infinite.
    To see the opposite implication assume $A\subset\RR$ is right infinite.
    Fix a $K>0$.
    We can find $c,d\in\RR^+$ such that $[c,d]\subset A$ and $d-c\geq 2K$.
    Let $a:=\max(K,c)$. 
    Then $K\leq a\leq K+c$ and $a+K\leq c+2K\leq d$. 
    Therefore, $[a,a+K]\subset[c,d]\subset A$, which proves \eqref{eq:inifinite-capacity-plus}.
    The proof of the remaining assertion is analogous. 
\qed

Let $N$ be an isolating neighborhood and let $\gamma$ be a full solution. 
We say that $\gamma$ that {\em wades through  $N$ in plus (minus) infinity} if there exists sequences $t^-,t^+:\NN\to\RR^+$
such that $t^-$, $t^+$ as well as $t^+-t^-$ tend to plus (minus) infinity and for all $n\in\NN$
\begin{equation}
\label{eq:wade}
 \gamma([t^-_n,t^+_n])\subset N.
\end{equation}
If $\gamma$ wades through  $N$ in plus (minus) infinity we write $\gamma\xrightarrow{+} N$ ($\gamma\xrightarrow{-} N$).

\begin{prop}
\label{prop:wade}
Assume $N$ is an isolating neighborhood and $\gamma$ is a full solution. 
Then the following conditions are mutually equivalent. 
\begin{itemize}
   \item[(i)] The preimage $\gamma^{-1}(N)$ is right (respectively left) infinite.
   \item[(ii)] Solution $\gamma$ wades through $N$ in plus (respectively minus) infinity.
   \item[(iii)] The intersection of $\Inv N$ and $\omega(\gamma)$ (respectively $\alpha(\gamma)$) is non-empty.
\end{itemize}
\end{prop}
\proof 
    We give the proofs that (i) is equivalent to (ii) in the case when $\gamma^{-1}(N)$ is right infinite. 
    The proof in the other case is analogous. 
    Assume (i) and set
    \[
        t^+(0):=t^-(0):=\inf\setof{s\geq 0\mid\gamma(s)\in N}.
    \]
    Proceeding recursively, assume $t^+_{n-1}$ and  $t^-_{n-1}$ are already defined for some $n\in\NN$.
    Using Proposition~\ref{prop:inifinite-capacity} for $K:=\max(n,t^+_{n-1})$
    select an $a\geq K$ such that $[a,a+K]\subset \gamma^{-1}(N)$ and set $t^-_n:=a$, $t^+_n:=a+K$.
    Then $t^+_n > t^-_n\geq n$, $t^+_n - t^-_n\geq n$ and $[t^-_n,t^+_n]\subset \gamma^{-1}(N)$.
    In particular, sequences $t^-$, $t^+$, $t^+-t^-$ tend to infinity
    Hence, $\gamma$ wades through $N$ in plus  infinity, which proves (ii).
    Obviously, (ii) implies (i). Hence the proof of the equivalence of (i) and (ii) is completed. 
    
    We now give the proofs that (ii) is equivalent to (iii) in the case when $\gamma$ wades through $N$ in plus infinity.
    The proof in the other case is analogous. 
    Assume (ii). Then we can choose sequences $t^-,t^+:\NN\to\RR^+$
    such that $t^-$, $t^+$ as well as $t^+-t^-$ tend to plus  infinity and \eqref{eq:wade} is satisfied for all $n\in\NN$.
    Set $t_n:=\frac{t^-_n+t^+_n}{2}$, $x_n:=\gamma(t_n)$ and $\gamma_n(t):=\gamma(t_n+t)$
    By taking a subsequence, if necessary, we may assume that
    $\lim_{n\to\infty}x_n=x_*$ for some $x_*\in \omega(\gamma)$. 
    Since $\gamma_n([t^-_n-t_n,t^+_n-t_n])=\gamma([t^-_n,t^+_n])\subset N$, we get $x_*\in\Inv N$.
    Hence, $\Inv N\cap\omega(\gamma)\neq\emptyset$, which proves (iii).
    
    Finally assume (iii). 
    As a consequence of \cite[Theorem 2.1]{Churchill1971},
    for each $n\in\NN$ we can choose $U_n$, an open neighborhood of $M$ such that 
    \begin{equation}
    \label{eq:predecomposition-iso-ne-1}
       u\in U_n\implies \varphi(u,[-n,n])\subset N.
    \end{equation}
    Since $\omega(\gamma)\cap M\neq\emptyset$, we can choose a $T_n > 2n$ such $\varphi(x,T_n)\in U_n$.
    Set $t^-_n:=T_n-n$ and $t^+_n:=T_n+n$. Then, $t^+_n>t^-_n>n$ and $t^+_n - t^-_n=2n$ proving
    that $t^+$, $t^-$ and $t^+ - t^-$ tend to infinity. Moreover, since $u_n:=\varphi(x,T_n)\in U_n$, we get from 
    \eqref{eq:predecomposition-iso-ne-1} that 
    \[
       \varphi(x,[t^-_n,t^+_n])=\varphi(\varphi(x,T_n),[-n,n])=\varphi(u_n,[-n,n])\subset N,
    \]
    which proves (ii).
\qed

\begin{thm}
\label{thm:predecomposition-iso-ne}
Assume $\cM=\setof{M_p\mid p\in\VSet}$ is a finite indexed family of mutually disjoint isolated invariant sets in $S$.
Then, the following conditions are equivalent.
\begin{itemize}
   \item[(i)] Family $\cM$ is a Morse predecomposition of $S$.
   \item[(ii)] There exists a family $\cN=\setof{N_p\mid p\in\VSet}$
of mutually disjoint isolating neighborhoods in $S$ such that $M_p=\Inv N_p$ and 
for every full solution $\gamma$ in $S$ there exist $p,q\in\VSet$ such that $\gamma\xrightarrow{-}N_p$ and $\gamma\xrightarrow{+}N_q$.
\end{itemize}
\end{thm}
\proof
Consider a finite family $\cM=\setof{M_p\mid p\in\VSet}$ of mutually disjoint isolated invariant sets in $S$
and assume (i). Since $\cM$ is finite and its elements are mutually disjoint, 
we can choose a family  $\cN=\setof{N_p\mid p\in\VSet}$  
of mutually disjoint isolating neighborhoods such that $M_p=\Inv N_p$ for every $p\in\VSet$.
Let $\gamma$ be a full solution in $S$.
Then, by the definition of Morse predecomposition, there exist $p,q\in\VSet$ such that $\gamma$ is a link from $M_p$ to $M_q$.
It follows from the definition of link and Proposition~\ref{prop:wade} that $\gamma\xrightarrow{-}N_p$ and $\gamma\xrightarrow{+}N_q$.
This proves (ii).
Assume in turn that there exists a family $\cN=\setof{N_p\mid p\in\VSet}$ which meets the requirements of (ii).
To prove that $\cM$ is a Morse predecomposition of $S$ we only need to verify that every full solution in $S$ is a link between elements of $\cM$,
because, by the main assumption, family $\cM$ consists of  mutually disjoint isolated invariant sets in $S$.
Hence, let $\gamma:\RR\to S$ be a full solution. Then, by (ii) there exist $p,q\in\VSet$ such that $\gamma\xrightarrow{-}N_p$ and $\gamma\xrightarrow{+}N_q$
and it follows from Proposition~\ref{prop:wade} that $\gamma$ is a link from $M_p$ to $M_q$.
\qed

As an immediate consequence of Theorem~\ref{thm:predecomposition-iso-ne} and Proposition~\ref{prop:wade} we get the following corollary.

\begin{cor}
\label{cor:predecomposition-iso-ne}
Assume $\cN=\setof{N_p\mid p\in\VSet}$ is a family
of mutually disjoint isolating neighborhoods in $S$ such that $\Inv N_p\neq\emptyset$ and 
for every full solution $\gamma$ in $S$ there exist $p,q\in\VSet$ such that $\gamma\xrightarrow{-}N_p$ and $\gamma\xrightarrow{+}N_q$.
Then $\cN$ is a Conley predecomposition of $S$. Moreover, there is an edge from $p$ to $q$ in the flow defined Conley model of $\cN^\bullet$
if and only if there is a full solution $\gamma$ in $S$ which wades through $N_p$ in minus infinity and through $N_q$ in plus infinity. 
\qed
\end{cor}

\section{Concluding remarks}\label{sec:concluding_remarks}

The proposed concept of Morse predecomposition is a natural generalization of Morse decomposition.
It provides insight into the internal structure of Morse sets. 
It is defined  in the classical setting of flows as well as in the combinatorial setting
of multivector fields. Therefore, it may be a useful tool in the algorithmic analysis of recurrent dynamics. 

Several other questions remain to be answered. 
A particularly interesting problem concerns the counterpart of Theorem~\ref{thm:mvf_consolidation} for flows, 
that is, whether every Morse predecomposition naturally induces a Morse decomposition via consolidation of cliques in the flow-defined Conley model.
The problem is open.

Another important issue which needs to be addressed is the algorithm for the construction of Morse predecompositions in the combinatorial setting.
Theorem~4.1 in~\cite{DJKKLM2019} shows that 
the algorithm for minimal Morse decompositions of combinatorial multivector fields reduces to the algorithm for strongly connected components in a directed graph. 
In the case of predecomposition one would search for possibly small strongly connected subsets of strongly connected components.
This may be achieved by computing the simple cycles in the graph.
However, unlike the decomposition case where every strongly connected component is automatically an isolated invariant set
(see~\cite[Theorem 4.1]{DJKKLM2019}), in the predecomposition case one has to take the smallest locally closed, 
$\cV$-compatible superset of the strongly connected set to guarantee isolation.  
Moreover, contrary to Morse decomposition, there may be many different minimal Morse predecompositions in the combinatorial setting (see Example \ref{ex:mvf-2}).
Therefore, to fully understand the structure of an isolated invariant set algorithm should be able to retrieve all minimal predecompositions. 

Finally, it would be interesting to investigate the relation between  Morse predecomposition in the combinatorial case and
Morse predecomposition in the classical case  in the spirit of the result for Morse decomposition \cite{MW2021}.

\bibliographystyle{abbrvnat}
\bibliography{bibliography}

\end{document}